\documentclass[reqno,11pt]{article}
\usepackage{graphicx, amsmath, amsthm}
\usepackage{amssymb}

\input{su2.sty}

\begin{document}

%%%%%%%%%%%%%%%%%%%%%%%%%%%%%%%%%%%%%%%%%%%%%%%%%%%%%%%%%%%%%%%%%%%%%%%%%%%%%%

\title{The effect of classical noise \\
on a quantum two-level system}
\author{Jean-Philippe Aguilar, Nils Berglund}
\date{}   

\maketitle

\begin{abstract}
\noindent
We consider a quantum two-level system perturbed by classical noise. The noise
is implemented as a stationary diffusion process in the off-diagonal matrix
elements of the Hamiltonian, representing a transverse magnetic field. We
determine the invariant measure of the system and prove its uniqueness. In the
case of Ornstein--Uhlenbeck noise, we determine the speed of convergence to the
invariant measure. Finally, we determine an approximate one-dimensional
diffusion equation for the transition probabilities. The proofs use both
spectral-theoretic and probabilistic methods.
\end{abstract}

\leftline{\small{\it Date.\/} May 7, 2008.}
\leftline{\small 2000 {\it Mathematical Subject Classification.\/} 
60h10, 35P15 (primary), 81Q15, 93E03 (secondary)
}
\noindent{\small{\it Keywords and phrases.\/}
spin $1/2$,
noise,
heat bath,
open systems, 
stochastic differential equations, 
Ornstein--Uhlenbeck process,
transition times,
transition probabilities, 
spectral gap,
diffusions on Lie groups, 
Stroock--Varadhan theorem, 
averaging,
renewal equations,
Laplace transforms
}  

%%%%%%%%%%%%%%%%%%%%%%%%%%%%%%%%%%%%%%%%%%%%%%%%%%%%%%%%%%%%%%%%%%%%%%%

\section{Introduction}
\label{sec_intro}

Noise is often used as a model for the effect of the environment (for instance
a heat bath) on a relatively small system. In general, it is difficult to prove
rigorously from first principles that a stochastic model indeed gives a good
approximation of the real dynamics. However, such a proof has been obtained in
a number of specific classical systems, see for
instance~\cite{FKM,SL77,EPR,RBT00,RBT02}. 

For quantum systems, the theory of effective stochastic models is not yet so
well-developed. One approach in which progress has been made in recent years is
the approach of repeated quantum interactions. If the quantum system is assumed
to interact during successive short time intervals with independent copies of a
quantum heat reservoir, one can under certain assumptions derive effective
equations involving quantum noises. See for
instance~\cite{AttalPautrat2006,BruneauJoyeMerkli2006,AttalJoye2007JFA,
AttalJoye2007JSP} for recent works in this direction. 

In the present work, we consider the intermediate situation of a classical noise
acting on a quantum system. We shall focus on the simplest possible quantum
system, namely a spin-$1/2$. The classical noise is realised by adding
off-diagonal stochastic processes in the system's Hamiltonian. This situation
can be realised, for instance, by letting the spin interact with a magnetic
field subject to small stochastic fluctuations, due to the magnet
creating the field being subject to weak random noise. 

When the spin is prepared, say, in the \lq\lq down\rq\rq\ state, and subjected
to a constant magnetic field in the $z$-direction, it will remain in the same
state for ever. If the direction of the magnetic field is allowed to fluctuate
in time, however, transitions to the  \lq\lq spin-up\rq\rq\ state become
possible. The aim of this work is to estimate how long it takes for these
transitions to occur, depending on the characteristics of the noise. 

This paper is organised as follows. Section~\ref{sec_res} contains a detailed
description of the model and all results. The main results are the following: 
Theorem~\ref{thm_res1} shows that under rather general assumptions on the
noise, the system admits a unique invariant measure induced by the Haar measure
on $\SU(2)$. Theorem~\ref{thm_res3} gives a spectral-gap estimate in the case
of Ornstein--Uhlenbeck noise, describing the speed of convergence to the
invariant measure, as a function of noise and coupling intensity, and of the
rate of decay of correlations of the noise term. Finally,
Theorem~\ref{thm_res2} uses an effective one-dimensional diffusion
approximation for the transition probability to derive expected transition
times. Sections~\ref{sec_puniq},~\ref{sec_pconv} and~\ref{sec_ptrans} contain
the proofs of these results. 

%%%%%%%%%%%%%%%%%%%%%%%%%%%%%%%%%%%%%%%%%%%%%%%%%%%%%%%%%%%%%%%%%%%%%%%

\section{Model and results}
\label{sec_res}

%%%%%%%%%%%%%%%%%%%%%%%%%%%%%%%%%%%%%%%%%%%%%%%%%%%%%%%%%%%%%%%%%%%%%%%

\subsection{Definition of the model}
\label{ssec_def}

The simplest possible quantum system is a two-level system (e.g. a spin
$1/2$), which can be described by the unperturbed Hamiltonian 
\begin{equation}
\label{def1}
H_0 = 
\begin{pmatrix}
\frac12 & 0 \\ 0 & -\frac12
\end{pmatrix}
\end{equation}
acting on the Hilbert space $\cH=\C^2$. 
We would like to perturb this system by \lq\lq classical noise\rq\rq, able
to induce transitions between the two energy levels. The noise should
however preserve the unitary character of the quantum evolution. A way to
do this is to add a non-diagonal interaction term to the Hamiltonian, of the
form $\kappa V(t)$, with
\begin{equation}
\label{def2}
V(t) = 
\begin{pmatrix}
0 & Z_t \\ \cc{Z_t} & 0
\end{pmatrix}\;, \qquad
Z_t = X_t + \icx Y_t\;,
\end{equation}
where $\set{Z_t}_{t\geqs0}$ is an ergodic, stationary
Markov process on some filtered probability space. This is equivalent to
assuming that the spin $1/2$ is interacting with a magnetic field
$B_t=(2\kappa X_t,2\kappa Y_t,1)$, since the total Hamiltonian can be
written 
\begin{equation}
H(t) = H_0 + \kappa V(t) = \kappa X_t\sigma^1 + \kappa Y_t\sigma^2 +
\frac12\sigma^3 \bydef \frac12 B_t \cdot \sigma\;,
\end{equation}
where the $\sigma^i$ are the Pauli matrices. For definiteness, we
shall assume that $Z_t$ is the solution of an It\^o
stochastic differential equation (SDE) of the form 
\begin{equation}
\label{def3a}
\6Z_t = f(Z_t) \6t + g(Z_t) \6W_t\;,
\end{equation}
where $\set{W_t}_{t\geqs0}$ denotes a standard Brownian motion, and $f$
and $g$ satisfy the usual Lipshitz and bounded-growth conditions ensuring
existence of a pathwise unique strong solution for any initial condition. 
A typical choice would be an Ornstein--Uhlenbeck process, defined by the 
SDE 
\begin{equation}
\label{def3}
\6Z_t = -\gamma Z_t \6t + \sigma \6W_t\;,
\end{equation}
where the initial condition $Z_0$ is a centred Gaussian of variance
$\sigma^2/2\gamma$. This represents the situation of the magnet generating the
field being subjected to a harmonic potential and white noise. 
This choice is also motivated by the fact that in certain
situations, the effect of a classical heat bath on a small (classical)
system has been rigorously shown to be describable by such an
Ornstein--Uhlenbeck process~\cite{EPR}. One can, however, consider other
types of stochastic processes as well, and part of our results do not
depend on the detailed definition of $Z_t$. Furthermore, one can easily
deal with more general interaction terms than~\eqref{def3}, including
several sources of noise for instance. However, the situation with a single
noise is in some sense the most interesting one, since it is in this
situation that noise-induced transitions are the most difficult. 

The evolution of the coupled system is governed by the
time-dependent Schr\"odinger equation 
\begin{equation}
\label{def4}
\icx \6\psi_t = H(t) \psi_t \6t\;.
\end{equation}
Provided the
stochastic process $Z_t$ has continuous sample paths, which are
stochastically bounded on compact time intervals\footnote{In other words,
$\lim_{L\to\infty} \prob{\sup_{0\leqs t\leqs T}\abs{Z_t}>L} = 0$ for all
$T>0$.} (as is the case for the
Ornstein--Uhlenbeck process), Dyson's theorem implies the existence of a
unique two-parameter family of unitary operators $U(t,s)\in\U(2)$, forming
a (strongly) continuous semi-group and such that 
\begin{equation}
\label{def5}
\psi_t = U(t,s)\psi_s
\end{equation}
holds almost surely for all $t > s$. In our case, $H(t)$ having zero trace,
we have in fact $U(t,s)\in\SU(2)$. For fixed $s$, the process $t\mapsto
U_t=U(t,s)$ satisfies the equation 
\begin{equation}
\label{def6}
\icx \6U_t = H(t) U_t \6t\;.
\end{equation}
Thus, characterising the evolution of the quantum system is equivalent to
characterising the stochastic process $\set{U_t}_{t\geqs0}$
on $\SU(2)$.

%%%%%%%%%%%%%%%%%%%%%%%%%%%%%%%%%%%%%%%%%%%%%%%%%%%%%%%%%%%%%%%%%%%%%%%

\subsection{Coordinates on $\SU(2)$}
\label{ssec_coord}

In order to study the solutions of Equation~\eqref{def6}, we have to
choose a convenient parametrisation of the Lie group $\SU(2)$. A possible
choice is to decompose elements of $\SU(2)$ on a basis of identity and
Pauli matrices as 
\begin{equation}
\label{coord1}
U = \icx(x_1\sigma^1+x_2\sigma^2+x_3\sigma^3) + x_4\one = 
\begin{pmatrix}
\phantom{-}x_4 + \icx x_3 & x_2 + \icx x_1 \\
 -x_2 + \icx x_1 & x_4 - \icx x_3
\end{pmatrix}\;,
\end{equation}
where the $x_i$ are real and satisfy 
\begin{equation}
\label{coord2}
x_1^2+x_2^2+x_3^2+x_4^2 = 1\;.
\end{equation}
Plugging into~\eqref{def6} and using It\^o's formula, one gets the
system of equations 
\begin{equation}
\label{coord3}
\begin{split}
\6x_{1,t} &= \bigbrak{-\tfrac12 x_{2,t} 
- \kappa X_t \, x_{4,t} - \kappa Y_t \, x_{3,t}}\6t\;, \\
\6x_{2,t} &= \phantom{-}\bigbrak{\tfrac12 x_{1,t} 
- \kappa X_t \, x_{3,t} + \kappa Y_t \, x_{4,t}}\6t\;, \\
\6x_{3,t} &= \bigbrak{-\tfrac12 x_{4,t} 
+ \kappa X_t \, x_{2,t} + \kappa Y_t \, x_{1,t}}\6t\;, \\
\6x_{4,t} &= \phantom{-}\bigbrak{\tfrac12 x_{3,t} 
+ \kappa X_t \, x_{1,t} - \kappa Y_t \, x_{2,t}}\6t\;.
\end{split}
\end{equation}
However, one of these equations is redundant because of
Condition~\eqref{coord2}. It is preferable to work with a
three-dimensional parametrisation of $\SU(2)$. A classical way to do
this is to write $U\in\SU(2)$ as the exponential $U=\e^{\icx M}$, where
$M$ lives in the Lie algebra $\su(2)$. Decomposing $M$ on the basis of
Pauli matrices as $M=r(\hat y_1\sigma^1+\hat y_2\sigma^2+\hat
y_3\sigma^3)$, with $\hat y_1^2+\hat y_2^2+\hat y_3^2=1$, 
yields the representation
\begin{equation}
\label{coord4}
U = 
\begin{pmatrix}
\cos r + \icx \hat y_3 \sin r &
\brak{\icx \hat y_1 + \hat y_2}\sin r \\
\brak{\icx \hat y_1 - \hat y_2}\sin r &
\cos r - \icx \hat y_3 \sin r 
\end{pmatrix}\;.
\end{equation}
One can then write the $\hat y_i$ in spherical coordinates in order to
obtain a three-dimensional system equivalent to~\eqref{def6}. However, the
resulting system turns out to have a rather complicated form. 

After some trials, one finds that the best suited parametrisation of
$\SU(2)$ for the present situation is given by 
\begin{equation}
\label{coord5}
U = u(\chi,\phi,\psi) \defby 
\begin{pmatrix}
\cos\chi \e^{-\icx(\phi/2+\psi)} & \sin\chi \e^{-\icx(\phi/2-\psi)} \\
-\sin\chi \e^{\icx(\phi/2-\psi)} & \cos\chi \e^{\icx(\phi/2+\psi)}
\end{pmatrix}\;.
\end{equation}
The variable $\chi$ lives in the interval $[0,\pi/2]$, while the pair
$(\phi/2+\psi,\phi/2-\psi)$ lives in the two-torus $\T^2$.\footnote{Hence
$(\phi,\psi)$ belong to a twisted two-torus (because
$u(\chi,\phi+2\pi,\psi)=u(\chi,\phi,\psi+\pi)$), but this will be of no
concern to us.}

\begin{prop}
\label{prop_coord}
The system~\eqref{def6} is equivalent to the system
\begin{equation}
\label{coord6}
\begin{split}
\6\chi_t &= \kappa \bigbrak{X_t\,\sin\phi_t + Y_t\,\cos\phi_t}
\,\6t\;,\\
\6\phi_t &= \biggpar{1 + \frac{2\kappa}{\tan 2\chi_t} 
\bigbrak{X_t\,\cos\phi_t - Y_t\,\sin\phi_t}} \6t\;,\\
\6\psi_t &= -\frac{\kappa}{\sin 2\chi_t} 
\bigbrak{X_t\,\cos\phi_t - Y_t\,\sin\phi_t} \6t\;.
\end{split}
\end{equation}
\end{prop}
\begin{proof}
This is a straightforward application of It\^o's formula. The second-order
term in the formula actually vanishes, due to the fact that the diffusion
term only enters indirectly, via the equation defining $Z_t$, and the
change of variables is independent of $Z$.
\end{proof}

Notice that the system~\eqref{coord6} has a skew-product structure, as the
right-hand side does not depend on $\psi_t$. Thus in fact all the
important information on the dynamics is contained in the first two
equations, while the evolution of $\psi_t$ is simply driven by
$(\chi_t,\phi_t)$ without any retroaction.

In the uncoupled case $\kappa=0$, we simply have $\dot\phi=1$ while
$\chi$ and $\psi$ are constant. Taking into account the initial
condition $U_0=\one$, we recover the fact that
\begin{equation}
U_t=
\begin{pmatrix}
\e^{-\icx t/2}&0\\0&\e^{\icx t/2}
\end{pmatrix} \;.
\end{equation}

%%%%%%%%%%%%%%%%%%%%%%%%%%%%%%%%%%%%%%%%%%%%%%%%%%%%%%%%%%%%%%%%%%%%%%%

\subsection{Invariant measure}
\label{ssec_resinv}

We first prove existence and uniqueness of an invariant measure for the
process $U_t$, under a rather general assumption on the noise $Z_t$.

\begin{assump}
\label{ass_Markov}
$\set{Z_t}_{t\geqs0}$ is a stationary, ergodic, real or
complex-valued Markov process on the (canonical) probability space
$(\Omega,\cF,\set{\cF_t}_{t\geqs0},\fP)$ of the Brownian motion
$\set{W_t}_{t\geqs0}$, defined by the It\^o SDE~\eqref{def3a},
admitting the unique invariant probability measure $\nu$.
\end{assump}

This assumption is satisfied if the drift term $f$ in the
SDE~\eqref{def3a} is sufficiently confining. For instance, in the
real case, it is sufficient to assume that $-Zf(Z)$ grows at least
quadratically as $\abs{Z}\to\infty$, and that $g(Z)$ satisfies some ellipticity
conditions. 

Recall that the Haar measure of a compact Lie group $G$ is the unique
probability measure $\mu$ on $G$ such that $\mu(gB)=\mu(B)$ for all Borel
sets $B\subset G$. A computation of the Jacobian of the
transformation~\eqref{coord5} shows that in our
system of coordinates, the Haar measure is given by 
\begin{equation}
\label{res1}
\mu(\6g) = \frac1{4\pi^2} \sin(2\chi) \,\6\chi\6\phi\6\psi\;.
\end{equation}

\begin{theorem}
\label{thm_res1}
Assume Assumption~\ref{ass_Markov} holds. Then for any $\kappa\neq0$ 
the product measure $\nu\otimes\mu$ is the unique invariant probability
measure of the process $(Z_t,U_t)$. 
\end{theorem}

The proof is given in Section~\ref{sec_puniq}. The fact that
$\nu\otimes\mu$ is invariant is checked by a straightforward computation,
while its uniqueness is proved using a control argument and the
Stroock--Varadhan theorem. 

\begin{remark}
\label{rem_res1}
Introducing a variable $\rho=\sin^2(\chi)$ allows to rewrite the system
in such a way that the invariant density is uniform. We will use this fact
later on in the proof.
\end{remark}

Theorem~\ref{thm_res1} leads to the following observations in terms of the
physics of the model. Let $|\pm\rangle$ denote the canonical basis vectors
of $\cH$, i.e., the \lq\lq spin up\rq\rq\ and \lq\lq spin down\rq\rq\
states. Assume for instance that the system is prepared in the $|-\rangle$
state at time $0$. Measuring the spin at time $t$, we will get $-1/2$
with probability 
\begin{equation}
\label{res2}
\bigabs{\langle - | U_t | - \rangle}^2 = \cos^2\chi_t
= 1 - \rho_t
\end{equation}
and $+1/2$ with probability 
\begin{equation}
\label{res2b}
\bigabs{\langle + | U_t | - \rangle}^2 = \sin^2\chi_t
= \rho_t\;.
\end{equation}
We should beware not to mix two different notions of probability. The
expression~\eqref{res2} represents the probability of a quantum
measurement yielding a certain value, for a given realisation
$\omega\in\Omega$ of the noise. This is independent of the probability
distribution on path space of the stochastic process $Z_t$. If the process
$U_t$ were in the stationary state $\mu$, the expected probability of
measuring a spin $+1/2$ would be given by 
\begin{equation}
\label{res3}
\bigexpecin{\mu}{\sin^2\chi} = 
\int_0^{2\pi} \int_0^{2\pi} \int_0^{\pi/2} 
\sin^2\chi \frac1{4\pi^2}\sin(2\chi)
\,\6\chi\6\phi\6\psi = \frac12\;. 
\end{equation}
Also, we would have for any $b\in[0,1]$
\begin{equation}
\label{res3B}
\bigprobin{\mu}{\abs{\langle + | U_t | - \rangle}^2 \leqs b} = 
\int_0^{2\pi} \int_0^{2\pi} \int_{\setsuch{\chi}{\sin^2\chi\leqs b}}
\frac1{4\pi^2}\sin(2\chi)
\,\6\chi\6\phi\6\psi = b\;. 
\end{equation}
However, since by definition $U_0=\one$, the system cannot be in the
stationary state $\mu$ at any given finite time $t$. At best, it can
approach $\mu$ exponentially fast as time increases --- to ensure such a
behaviour, we have to show the existence of a spectral gap. We shall do
this under a more restrictive assumption on the noise term.

%%%%%%%%%%%%%%%%%%%%%%%%%%%%%%%%%%%%%%%%%%%%%%%%%%%%%%%%%%%%%%%%%%%%%%%

\subsection{Convergence to the invariant measure}
\label{ssec_rescon}

\begin{assump}
\label{ass_OU}
The process $\set{Z_t}_{t\geqs0}$ is the real-valued
stationary Ornstein--Uhlenbeck process defined by the SDE~\eqref{def3}. 
In other words, 
\begin{equation}
\label{res4}
Z_t = Z_0 \e^{-\gamma t} + \sigma \int_0^t \e^{-\gamma(t-s)} \6W_s\;,
\end{equation}
where $Z_0$ is a centred, Gaussian random variable, of variance
$\sigma^2/2\gamma$, which is independent of the Brownian motion
$\set{W_t}_{t\geqs0}$. 
\end{assump}

Then the invariant measure $\nu$ is also Gaussian,
centred, of variance $\sigma^2/2\gamma$. The parameter $\sigma^2$ can be
interpreted as the temperature of the heat bath creating the noise, while
$\gamma$ represents the rate of decay of correlations in the bath.

The assumption that $Z_t$ be real is not essential, but it simplifies the
notations. In effect, we have to study the SDE 
\begin{equation}
\label{res5}
\begin{split}
\6Z_t &= -\gamma Z_t\,\6t + \sigma\6W_t\;,\\
\6\chi_t &= \kappa Z_t\,\sin\phi_t\,\6t\;,\\
\6\phi_t &= \biggpar{1 + \frac{2\kappa Z_t\,\cos\phi_t}{\tan 2\chi_t}}
\6t\;,\\
\6\psi_t &= -\frac{\kappa Z_t\,\cos\phi_t}{\sin 2\chi_t}\,\6t\;,
\end{split}
\end{equation}
with initial condition $(\chi_0,\phi_0)=(0,0)$. The last equation, for
$\6\psi_t$, is not really important for the dynamics, and has no
measurable effect on physics either, $\psi_t$ having only an influence on
the phase of matrix elements of $U_t$.

With the process $X_t=(Z_t,\chi_t,\phi_t,\psi_t)$ we associate in the usual way
the Markov semigroup $T_t:\varphi(\cdot)\mapsto
\econd{\varphi(X_t)}{X_0=\cdot}$. Its infinitesimal generator is the
differential operator 
\begin{equation}
\label{res6}
L = L_Z + \kappa Z\,\sin\phi \dpar{}{\chi} 
+ \biggpar{1 + \frac{2\kappa Z\,\cos\phi}{\tan 2\chi}} \dpar{}{\phi} 
- \frac{\kappa Z\,\cos\phi}{\sin 2\chi} \dpar{}{\psi}\;,
\end{equation}
where 
\begin{equation}
\label{res7}
L_Z = -\gamma Z \dpar{}{Z} + \frac{\sigma^2}{2} \dpar{^2}{Z^2}
\end{equation}
is the generator of the Ornstein--Uhlenbeck process. The invariant measure
$\nu\otimes\mu$ is an eigenfunction of the adjoint $L^*$ of $L$, with eigenvalue
$0$. It is known that $L_Z$ has discrete spectrum, with real, nonpositive
eigenvalues $-n\gamma$, $n\in\N_0$ (the eigenfunctions are Hermite polynomials).
Since the other variables $\chi$ and $\phi$ live in a compact set, $L$ also has
discrete spectrum. Moreover, Theorem~\ref{thm_res1} shows that $0$ is a simple
eigenvalue of $L$. 

\begin{theorem}
\label{thm_res3}
Assume that Assumption~\ref{ass_OU} holds. Then for sufficiently small
$\kappa\sigma$, 
there exists a constant $c$, independent of $\gamma$, $\kappa$ and
$\sigma$, such that all eigenvalues of $L$ except $0$ have a real part
bounded above by\footnote{For two real numbers $a,b$, $a\vee b$ denotes the
maximum of $a$ and $b$ and $a\wedge b$ denotes the minimum of $a$ and $b$.}
\begin{equation}
\label{res7B}
-\frac{c}{T_{\gamma,\kappa\sigma}} 
\qquad
\text{where }
T_{\gamma,\kappa\sigma}
=\frac{1+\gamma^2}{(\kappa\sigma)^2}\vee\frac1{\gamma}\;.
\end{equation}
\end{theorem}

The proof is given in Section~\ref{sec_pconv}. It relies on standard
second-order perturbation theory.

Theorem~\ref{thm_res3} shows that the distribution of the actual process
$U_t$, starting with initial condition $\one$, converges exponentially
fast to $\mu$, with rate $c/T_{\gamma,\kappa\sigma}$.
Thus the expected (quantum) probability to measure a value $+1/2$ for the spin,
when it starts in the \lq\lq down\rq\rq\ state satisfies 
\begin{equation}
\label{res8}
\bigexpecin{\delta_\one}{\sin^2\chi_t} = 
\frac12 \bigbrak{1 - \Order{\e^{-ct/T_{\gamma,\kappa\sigma}}}}\;. 
\end{equation}
Hence for times $t\gg T_{\gamma,\kappa\sigma}$, one can expect that
a measurement of the spin will yield $-1/2$ or $+1/2$ with probability
$1/2$ --- provided the result is averaged over many realisations of the
noise.

Regarding the dependence of the result on the correlation decay rate $\gamma$,
we can distinguish two asymptotic regimes:
\begin{itemiz}
\item	For $\gamma\ll1$, the correlations in the Ornstein--Uhlenbeck
process decay very slowly, in fact $Z_t$ resembles a Brownian motion. As a
result, the relaxation time $T_{\gamma,\kappa\sigma}$ becomes very large.
\item	For $\gamma\gg\kappa\sigma$, the correlations in the
Ornstein--Uhlenbeck process decay quickly, so that $Z_t$ resembles white noise.
Then the relaxation time $T_{\gamma,\kappa\sigma}$ also becomes large.
\end{itemiz}
Relaxation to equilibrium is fastest when $\gamma\simeq(\kappa\sigma)^2$. Then
$T_{\gamma,\kappa\sigma}$ has order $1/(\kappa\sigma)^2$.

%%%%%%%%%%%%%%%%%%%%%%%%%%%%%%%%%%%%%%%%%%%%%%%%%%%%%%%%%%%%%%%%%%%%%%%

\subsection{Diffusion approximation for the transition probability}
\label{ssec_reseff}

The probability of measuring a \lq\lq spin up\rq\rq\ at time $t$, if the
system is prepared in the \lq\lq spin down\rq\rq\ state at time $0$, for a
given realisation $\omega$ of the noise, is given by
\begin{equation}
\label{reseff0}
\rho_t(\omega) = \abs{\langle - | U_t(\omega) | + \rangle}^2 
= \sin^2(\chi_t(\omega))\;.
\end{equation}
While the spectral gap result Theorem~\ref{thm_res3} yields a control on the
average of $\rho_t$ over many realisations $\omega$ of the noise, it does not
describe its time-evolution very precisely. In this section we provide a more
precise description of the dynamics by giving pathwise estimates on
$\rho_t(\omega)$, in the form of first-passage times. This is done by obtaining
an approximately closed effective equation for $\rho_t$, which is the only
physically measurable quantity in the system. The methods used are partly
adapted from those presented in~\cite{BGbook}.

\begin{figure}
\centerline{\includegraphics*[clip=true,height=60mm]{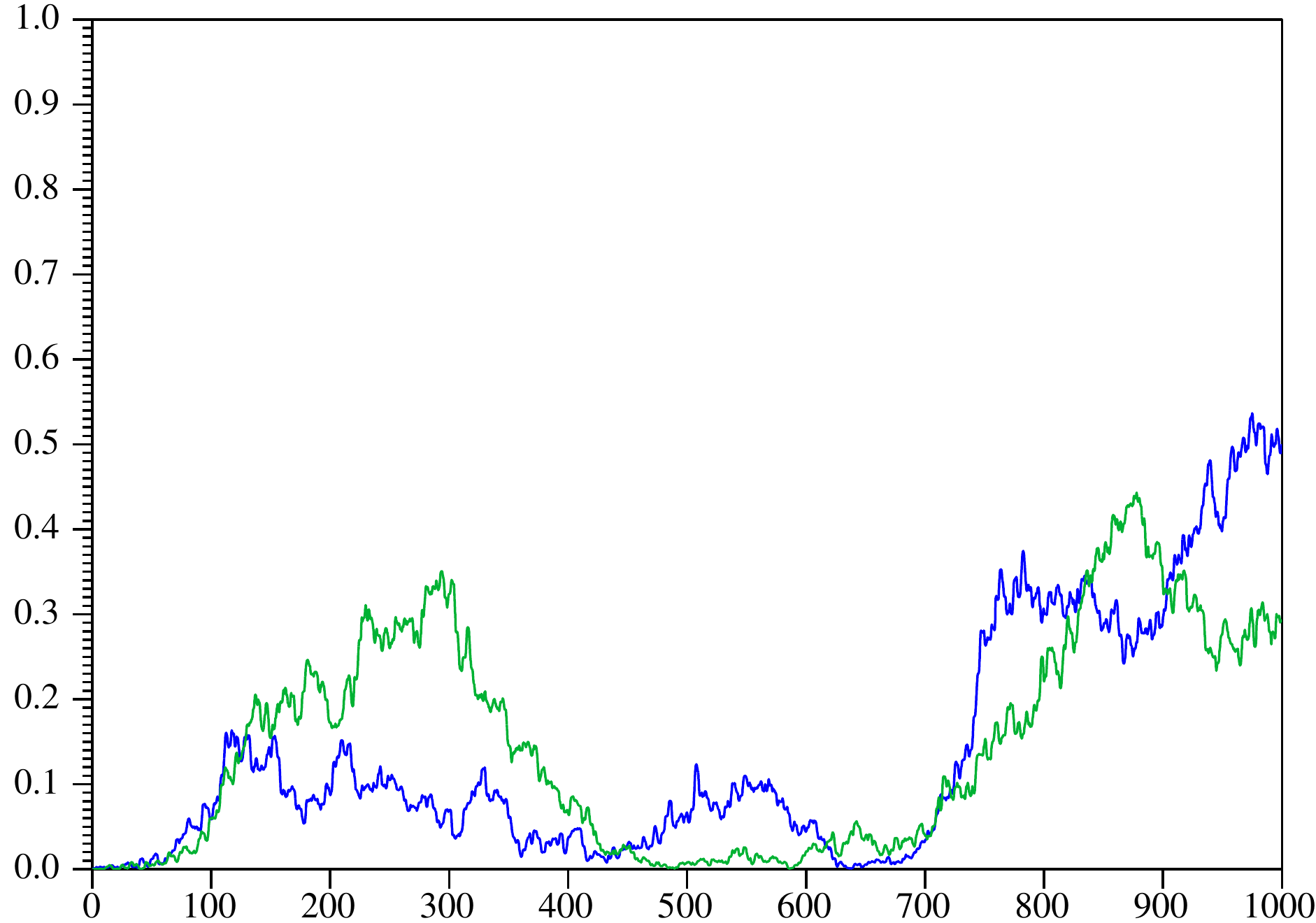}}
\caption[]{The probability $\abs{\langle - | U_t | +
\rangle}^2 = \sin^2(\chi_t)$ of measuring a transition from spin $-1/2$ to
$+1/2$ as a function of time $t$, for two different realisations of the noise.
Parameter values are $\gamma=1$, $\kappa=1$ and $\sigma=0.03$, so that the
relaxation time $T_{\gamma,\kappa\sigma}$ to the invariant measure is of order
$2000$.}
\label{fig1}
\end{figure}

\figref{fig1} shows two sample paths $t\mapsto\rho_t(\omega)$,
obtained for two different realisations $W_t(\omega)$ of the Brownian
motion driving $Z_t$. For the parameter values
$\kappa=\gamma=1$ and $\sigma=0.03$ used in \figref{fig1}, this probability
remains larger than $1/2$ for all times up to $1000$. Hence the system is
still far from its invariant measure.

\begin{figure}
\centerline{\includegraphics*[clip=true,height=60mm]{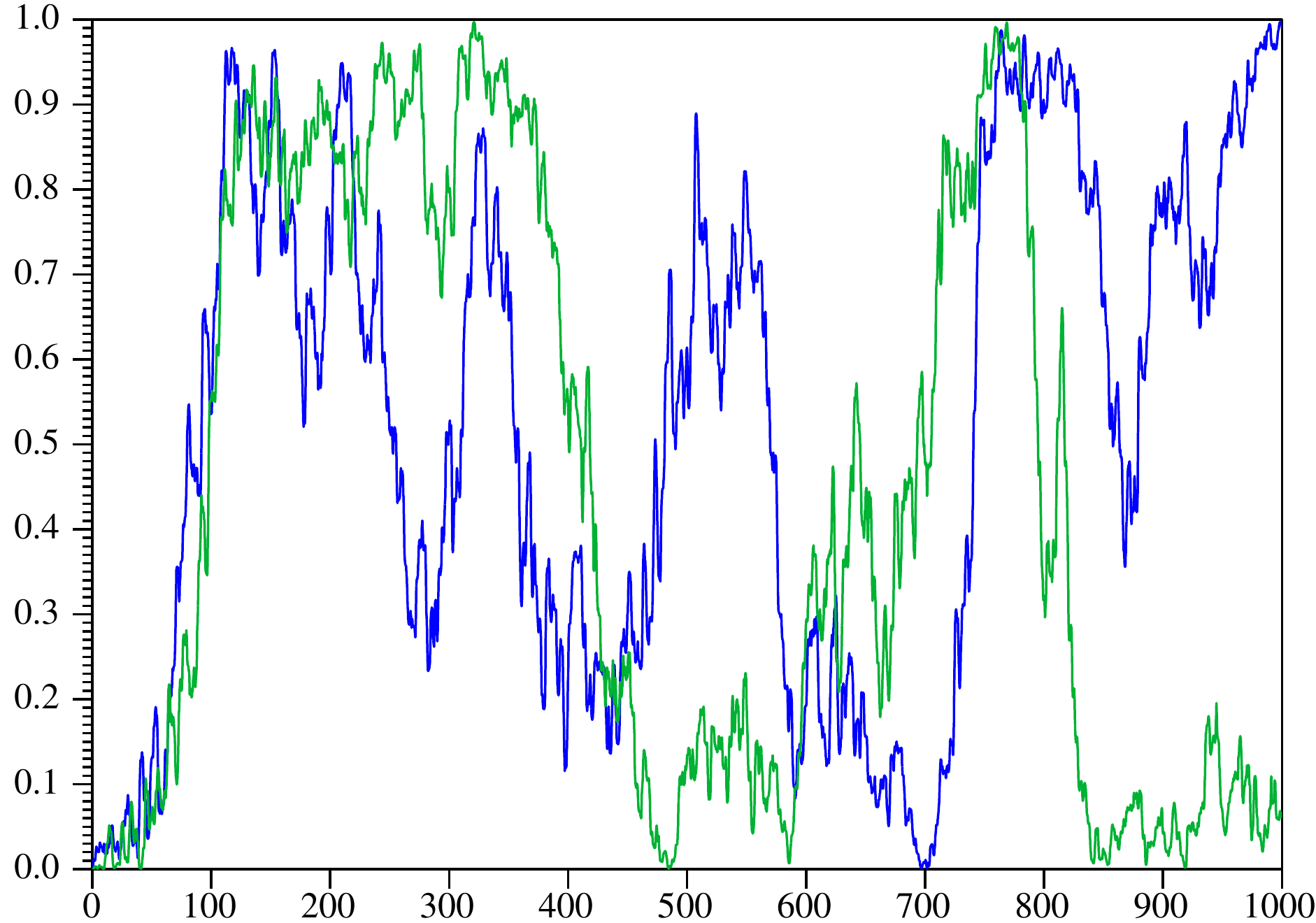}}
\caption[]{The probability $\abs{\langle - | U_t | +
\rangle}^2 = \sin^2(\chi_t)$ of measuring a transition from spin $-1/2$ to
$+1/2$ as a function of time $t$, for two different realisations of the noise.
Parameter values are $\gamma=1$, $\kappa=1$ and $\sigma=0.1$, so that the
relaxation time $T_{\gamma,\kappa\sigma}$ to the invariant measure is of order
$200$.}
\label{fig2}
\end{figure}

\figref{fig2} shows two sample paths $t\mapsto\rho_t(\omega)$ in a
case with larger noise intensity $\sigma=0.1$. Now the sample paths have
enough time to explore all of phase space, indicating that the system has
reached equilibrium. Note however that for any given realisation of the
noise, a spin measurement made after a sufficiently long time may still
yield $-1/2$ with any probability. Only by making repeated measurements in
the course of time would one obtain an average probability close to $1/2$
--- assuming that the measurements do not affect the state of the system,
which of course they do. 

Let us now describe the derivation of the effective equation for $\rho_t$. 
Recall that the uniform measure is invariant when the system is written in
the variables $(\rho,\phi,\psi)$. In order to exploit symmetries of the
problem, it is convenient to use the variable $y = 2\rho-1 =
-\cos2\chi\in[-1,1]$ instead of $\rho$. We are led to consider the SDE on
$\R\times[-1,1]\times\fS^1$
\begin{equation}
\label{reseff1}
\begin{split}
\6Z_t &= -\gamma Z_t\,\6t + \sigma\,\6W_t\;,\\
\6y_t &= 2\kappa Z_t\,\sqrt{1-y_t^2}\sin\phi_t\,\6t\;,\\
\6\phi_t &= \biggbrak{1 - 2\kappa Z_t\,\frac{y_t}{\sqrt{1-y_t^2}}
\cos\phi_t}  \6t\;.
\end{split}
\end{equation}
Observe that if $\kappa$ (or $\sigma$) is small, then this system displays
several distinct timescales. While $\phi_t$ will be close to $t$, $y_t$ can
at best grow like $\kappa t$. In fact, if $Z$ were constant, one can see that
the system~\eqref{reseff1} restricted to $(y,\phi)$ is conservative, that is,
there exists a constant of motion $K(y,\phi)=y+\Order{\kappa}$. The fact that
$Z_t$ is time-dependent actually destroys this invariance, but in a very soft
way. 

The usual way to obtain an effective diffusion equation for $y$ (or for the
constant of motion $K(y,\phi)$) is to average the right-hand side
of~\eqref{reseff1} over the fast variable $\phi$~\cite{FreidlinWeber04}.
However, in the present case this average is zero, so that one has to go beyond
the usual averaging procedure. 

It turns out that one can construct a new variable $\bar y=
y+\kappa w(Z,y,\phi)$ such that 
\begin{equation}
\label{reseff2}
\6\bar y_t \simeq -\frac{4\sigma^2\gamma}{1+\gamma^2} Z_t^2 \bar y_t \,\6t 
+ \frac{2\kappa\sigma}{\sqrt{1+\gamma^2}}\sqrt{1-\bar y_t^2}
\,\cos(\phi_t+\theta) \,\6W_t\;,
\end{equation}
where $\theta$ is a constant phase shift (see Proposition~\ref{prop_avrg} for
a precise formulation). The variance of the noise term grows like $t$ times the
square of the coefficient of $\6W_t$. Since $\phi_t$ rotates rapidly, we expect
that $\cos^2(\phi_t+\theta)$ can be approximated by $1/2$. In addition, $Z_t^2$
is rapidly fluctuating with average $\sigma^2/2\gamma$. We thus expect that 
\begin{equation}
\label{reseff3}
\6\bar y_t \simeq -\frac{2(\kappa\sigma)^2}{1+\gamma^2} \bar y_t
\,\6t 
+ \sqrt{\frac{2(\kappa\sigma)^2}{1+\gamma^2}}\sqrt{1-\bar y_t^2}
\,\6W_t\;,
\end{equation}
which, after rescaling time by a factor $T_{\gamma,\kappa\sigma}$ reduces to 
\begin{equation}
\label{reseff4}
\6\bar y_t \simeq - \bar y_t\,\6t 
+ \sqrt{1-\bar y_t^2}\,\6W_t\;.
\end{equation}
The sample paths of this system explore phase space in a rescaled time of order
$1$. 

The approximation~\eqref{reseff3} is of course not rigorous. But based on the
above intuition, one can prove the following result on the time needed for
$y_t$ to explore phase space. 

\begin{theorem}
\label{thm_res2}
Assume that Assumption~\ref{ass_OU} holds. Let 
\begin{equation}
\label{res5A}
\tau(y) = \inf\setsuch{t>0}{y_t>y} 
\end{equation}
be the random first-passage time of $y_t$ at the value $y$. Then
there is a function $e(y)$, independent of $\gamma$, $\sigma$ and
$\kappa$ and bounded for $y<1$, such that 
\begin{equation}
\label{res5B}
\bigexpec{\tau(y)} \leqs e(y) T_{\gamma,\kappa\sigma}
= e(y) 
\biggbrak{\frac{1+\gamma^2}{(\kappa\sigma)^2}\vee \frac1\gamma}\;.
\end{equation}
Furthermore, $\prob{\tau(y)>t}$ decays exponentially with rate of order $1/e(y)
T_{\gamma,\kappa\sigma}$. 
\end{theorem}

The proof is given in Section~\ref{sec_ptrans}. It uses comparison inequalities
for the stochastic differential equation satisfied by $\bar y_t$. The main
difficulty is that one obtains different approximations in different regions of
phase space, which have then to be patched together. This is done with the help
of renewal equations, solved by Laplace transforms.

The time $T_{\gamma,\kappa\sigma}$
plays the r\^ole of a typical exploration time. It has value $2000$
in~\figref{fig1} and $200$ in~\figref{fig2}. The function $e(y)$ may
diverge as $y\nearrow1$. This is due to the fact that $y=1$ (i.e., $\chi=\pi/2$)
corresponds to a single curve in $\SU(2)$, depending only on $\phi/2-\psi$,
a too small set to hit.

\begin{remark}\hfill
\begin{enum}
\item	Theorem~\ref{thm_res2} does not follow from Theorem~\ref{thm_res3}. For
instance, the Ornstein--Uhlenbeck process $Z_t$ has a spectral gap $\gamma$,
but the process takes exponentially long times of order $\e^{\gamma h/\sigma^2}$
to reach values of order $h$. 
\item	Conversely, Theorem~\ref{thm_res3} partially follows from
Theorem~\ref{thm_res2}, namely for real eigenvalues. Indeed, given a subset
$\cD$ of phase space, consider the following boundary value problem:
\begin{align}
\nonumber
&& -L\varphi &= \lambda\varphi &&\text{for $x\in \cD$\;,} &\\
&& \varphi &= g &&\text{for $x\in\partial \cD$\;.} &
\label{DV1}
\end{align}
For Dirichlet boundary conditions $g=0$, it is known that the eigenvalue
$\bar\lambda(\cD)$ of smallest real part of this problem is real and
simple~\cite{Jentzsch1912,ProtterWeinberger72}. A classical result due to
Donsker and Varadhan~\cite{DonskerVaradhan76} states that 
\begin{equation}
\label{DV2}
\bar\lambda(\cD) \geqs \frac1{\sup_{x\in D}{\expecin{x}{\tau_\cD}}}\;,
\end{equation}
where $\tau_\cD=\inf\setsuch{t>0}{x_t\not\in\cD}$ is the first-exit time from
$\cD$. Now if $\lambda<\bar\lambda(\cD)$, the boundary value problem~\eqref{DV1}
has a unique solution for every boundary condition $g$, given by 
\begin{equation}
\label{DV3}
\varphi(x) = \bigexpecin{x}{\e^{\lambda\tau_\cD} g(x_{\tau_\cD})}\;.
\end{equation}
Any eigenfunction of $L$ corresponding to a nonzero real eigenvalue must change
sign. Thus taking $\cD=\setsuch{x}{\varphi(x)>0}$ would yield $g=0$, and thus
also $\varphi(x)=0 \;\forall x\in\cD$, a contradiction. Hence $L$ cannot have
any real eigenvalue smaller than $\bar\lambda(\cD)$, except $0$. 
\end{enum}
\end{remark}

%%%%%%%%%%%%%%%%%%%%%%%%%%%%%%%%%%%%%%%%%%%%%%%%%%%%%%%%%%%%%%%%%%%%%%%

\section{Existence and uniqueness of the invariant measure}
\label{sec_puniq}

%%%%%%%%%%%%%%%%%%%%%%%%%%%%%%%%%%%%%%%%%%%%%%%%%%%%%%%%%%%%%%%%%%%%%%%

\subsection{Invariance of the Haar measure}
\label{ssec_Haar}

With the semigroup $T_t$ of the Markov process, we associate the dual semigroup
$S_t$, acting on $\sigma$-finite measures~$\mu$ according to
\begin{equation}
\label{Haar1}
(S_t\mu)(B) = \int \bigpcond{X_t\in B}{X_0=x} \,\mu(\6x) 
\bydef \bigprobin{\mu}{X_t\in B}\;.
\end{equation}
Its infinitesimal generator is the adjoint $L^*$ of the generator $L$. If $\mu$
has density $\rho$ with respect to the Lebesgue measure, then 
\begin{equation}
\label{Haar2}
L^*\rho = 
L_Z^*\rho - \kappa Z \sin\phi \dpar{\rho}{\chi}
- \dpar{}{\phi}\biggbrak{1+\biggpar{\frac{2\kappa Z\cos\phi}{\tan2\chi}}\rho}
+ \frac{\kappa Z\cos\phi}{\sin 2\chi} \dpar{\rho}{\psi}\;.
\end{equation}
Now in our case, the measure which is claimed to be invariant is the product
measure $\nu\otimes\mu$, where $\mu$ has density $\sin(2\chi)/4\pi^2$. Since
$\nu$ is invariant for the process $Z_t$, $L_Z^*\nu=0$. It is then immediate to
check that $L^*\rho=0$.
\qed

%%%%%%%%%%%%%%%%%%%%%%%%%%%%%%%%%%%%%%%%%%%%%%%%%%%%%%%%%%%%%%%%%%%%%%%

\subsection{Proof of Theorem~\ref{thm_res1}}
\label{ssec_puniq}

We would like to show that $\nu\otimes\mu$ is the unique invariant measure of
the process $(Z_t,U_t)$. Since the generator $L$ of the process is very far
from uniformly elliptic, we cannot use standard ellipticity (or even 
hypo-ellipticity) arguments. To circumvent this difficulty, we shall combine an
argument of control theory with the Stroock--Varadhan support theorem.

We start by writing the joint system~\eqref{def3a},\eqref{coord6} as 
\begin{align}
\nonumber
\6Z_t &= f(Z_t)\,\6t + g(Z_t)\,\6W_t\;, \\
\6x_t &= \bigbrak{b_0(x_t) + \kappa X_t b_1(x_t) + \kappa Y_t b_2(x_t)}\,\6t\;,
\label{puniq1}
\end{align}
where $Z_t=X_t+\icx Y_t$, $x_t=(\chi_t,\phi_t,\psi_t)$ and 
\begin{align}
\nonumber
b_0(x) &= (0,1,0)\;, \\
\nonumber
b_1(x) &= \biggpar{\sin\phi, \frac{2\cos\phi}{\tan 2\chi},
-\frac{\cos\phi}{\sin2\chi}}\;, \\
b_2(x) &= \biggpar{\cos\phi, -\frac{2\sin\phi}{\tan 2\chi},
\frac{\sin\phi}{\sin2\chi}}\;. 
\label{puniq2}
\end{align}
We denote by $P_t((Z_0,x_0),B) = \probin{(Z_0,x_0)}{(Z_t,x_t)\in B}$ 
the transition probabilities of the Markov process with initial condition
$(Z_0,x_0)$. With~\eqref{puniq1} we associate the control system 
\begin{align}
\nonumber
\dot Z &= f(Z) + g(Z)\,u(t)\;, \\
\dot x &= b_0(x) + \kappa X b_1(x) + \kappa Y b_2(x)\;,
\label{puniq3}
\end{align}
where $u\colon\R_+\to\R$ is a piecewise constant function. The \emph{accessible
set}\/ from an initial condition $(Z_0,x_0)$ is the set 
\begin{equation}
\label{puniq4}
A_t(Z_0,x_0) = \bigsetsuch{(Z,x)}{\exists u\colon[0,t]\to\R, 
(Z(0),x(0))=(Z_0,x_0), (Z(t),x(t))=(Z,x)}\;.
\end{equation}
The Stroock--Varadhan support
theorem~\cite{StroockVaradhan72,StroockVaradhanProc72} states that 
\begin{equation}
\label{puniq5}
\supp P_t((Z_0,x_0),\cdot)  
= \overline{A_t(Z_0,x_0)}\;,
\end{equation}
that is, that $P_t((Z_0,x_0),B)>0$ for any open set $B\subset A_t(Z_0,x_0)$.

\begin{prop}
For any $\kappa\neq0$, any initial condition $(Z_0,x_0)$ and any $t>0$, the
closure of the accessible set $A_t(Z_0,x_0)$ is the whole phase space
$\C\times[0,\pi/2]\times\T^2$. 
\end{prop}
\begin{proof}
We first observe that by Assumption~\ref{ass_Markov}, the process $Z_t$ is
controllable, that is, for any $t>0$ and any trajectory $\set{Z_s}_{0\leqs
s\leqs t}$, one can find a piecewise constant control $u$ such that
$f(Z_t)+g(Z_t)u(t)$ approximates $\dot Z(t)$ arbitrarily closely. As a
consequence, we can slightly simplify the problem by considering the
three-dimensional control problem 
\begin{equation}
\label{puniq6}
\dot x = b_0(x) + \kappa b_1(x) u_1(t) + \kappa b_2(x) u_2(t)\;. 
\end{equation}
It is in fact sufficient to discuss the particular case $Z_t\in\R$, that is,
$u_2=0$, the other cases being treated similarly. 

Let $\Gamma=[0,\pi/2]\times\T^2$. 
The \emph{accessibility algebra}\/ $\cA(x)$ of a point $x\in\Gamma$ is the
smallest sub-algebra of the Lie algebra of vector fields over
$\Gamma$ containing $b_0$ and $b_1$. It is generated by all iterated Lie
brackets of $b_0$ and $b_1$. 
We identify each vector field $b_i$ with an operator $A_i=\sum_j
b^j_i\partial_j$. Then Lie brackets of $b_i$ are identified with commutators of
$A_i$. Now straightforward computations show that 
\begin{align}
\nonumber
[A_0,A_1] &= A_2\;,\\
\nonumber
[A_0,A_2] &= -A_1\;,\\
[A_1,A_2] &= 4 A_0\;.
\label{puniq7}
\end{align}
As a consequence, the accessibility algebra $\cA(x)$ is generated by the vector
fields $b_0$, $b_1$ and $b_2$. Since 
\begin{equation}
\label{puniq8}
\det\set{b_0,b_1,b_2} = -\frac{1}{\sin2\chi}\;,
\end{equation}
we conclude that $\Span[\cA(x)]$ has dimension $3$ for any $x$, except possibly
for $\chi\in\set{0,\pi/2}$. 

In order to treat the case $\chi=\pi/2$, we use the
representation~\eqref{coord3} of the system. $\chi=\pi/2$ corresponds to
$x_3=x_4=0$. Near a point $(x^\star_1,x^\star_2,0,0)$, the Lie group $\SU(2)$
can be parametrised by $(y_2,y_3,y_4)=(x_2-x^\star_2,x_3,x_4)$, the remaining
variable $x_1$ being expressed in terms of the other variables with the help of
Relation~\eqref{coord2}. Writing the vector field in variables $(y_2,y_3,y_4)$
and computing Lie brackets, we obtain again that the accessibility algebra is
generated by three vector fields and has dimension $3$. The case $\chi=0$ is
treated similarly. 

Whenever $\Span[\cA(x)]$ has full dimension, a standard result
from control theory (see \cite[Chapter~3]{NijmeijerVanderSchaft90}) states that
$A_t(x)$ contains an open, non-empty neighbourhood of $x$. Since this is true
for all $x$, we have proved that $\overline{\cA(x)}=\Gamma$ for all
$x\in\Gamma$.
\end{proof}

\begin{proof}[{\sc Proof of Theorem~\ref{thm_res1}}]
Let $\cM$ denote the convex set of all invariant measures of the process. 
With every invariant measure $\mu\in\cM$ we can associate a stationary Markov
process $\fP^\mu$. It is known that a measure $\mu\in\cM$ is extremal if and
only if the associated dynamical system is ergodic. Thus if $\cM$ contains two
measures $\mu_1\neq\mu_2$, the measure $\mu=\frac12(\mu_1+\mu_2)$ is invariant
but not ergodic. 

The previous proposition implies, with Stroock--Varadhan's support theorem,
that $P_t((Z_0,x_0),B)>0$ for any open $B\subset\C\times[0,\pi/2]\times\T^2$. 
If $\mu$ were not ergodic, we could find an open set $B$, with $0<\mu(B)<1$,
such that $P_t((Z_0,x_0),B)=1$ for $\mu$-almost all $(Z_0,x_0)\in B$ and
$P_t((Z_0,x_0),B^c)=0$ for $\mu$-almost all $(Z_0,x_0)\in B^c$, a contradiction.
Hence $\mu$ must be ergodic, and thus $\cM$ contains only one measure.
\end{proof}

%%%%%%%%%%%%%%%%%%%%%%%%%%%%%%%%%%%%%%%%%%%%%%%%%%%%%%%%%%%%%%%%%%%%%%%

\section{Convergence to the invariant measure}
\label{sec_pconv}

Let us first remark that since the Ornstein--Uhlenbeck process $Z_t$ is
proportional to $\sigma$, and enters in the dynamics only through the coupling
constant $\kappa$, all results will only depend on the product $\kappa\sigma$.
We may thus choose one of the parameters at our convenience, keeping the
product fixed. In this section, we consider that in fact $\sigma=1$, but keep
writing $\sigma$ for more clarity.  

%%%%%%%%%%%%%%%%%%%%%%%%%%%%%%%%%%%%%%%%%%%%%%%%%%%%%%%%%%%%%%%%%%%%%%%

\subsection{Proof of Theorem~\ref{thm_res3}}
\label{ssec_gap}

In order to exploit the fact that the invariant measure is uniform when
taking variables $(y=2\sin^2\chi-1,\phi,\psi)$, we consider, instead of the
infinitesimal generator~\eqref{res6} of our diffusion process, the equivalent
operator  
\begin{equation}
\label{gap01}
L = L_Z + \kappa Z
\sqrt{1-y^2}\,\sin\phi \dpar{}{y} 
+ \biggbrak{1 - \frac{2\kappa Z y}{\sqrt{1-y^2}}\,\cos\phi} \dpar{}{\phi} 
- \frac{\kappa Z\,\cos\phi}{\sqrt{1-y^2}} \dpar{}{\psi}\;.
\end{equation}
We can write this generator as $L=L_0 + \icx \kappa L_1$, where
$L_0=L_Z+\tdpar{}{\phi}$. The operator $L_Z$ is self-adjoint in
$L^2(\R,\nu(\6Z))$, and has simple eigenvalues $-n\gamma$, $n\in\N_0$. Its
eigenfunctions are (properly normalised) Hermite polynomials, which we denote
$h_n(Z)$. In particular, we have 
\begin{equation}
\label{gap02B}
h_0(Z) = 1\;,
\qquad
h_1(Z) = \frac{\sqrt{2\gamma}}{\sigma} Z\;.
\end{equation}
An integration by parts shows that 
\begin{equation}
\label{gap02}
L_1 \defby -\icx Z \sqrt{1-y^2}\,\sin\phi \dpar{}{y} 
+ \frac{2 \icx Z y}{\sqrt{1-y^2}}\,\cos\phi \dpar{}{\phi} 
+ \frac{\icx Z\,\cos\phi}{\sqrt{1-y^2}} \dpar{}{\psi}
\end{equation}
is also self-adjoint in $L^2([-1,1]\times\T^2,\6y\6\phi\6\psi)$. We may choose
an orthonormal basis of $L^2(\R\times[-1,1]\times\T^2,(\nu\otimes\mu)(\6x))$
given by 
\begin{equation}
\label{gap03}
\ket{n,p,k,r} = h_n(Z)f_p(y)\e^{\icx k\phi}\e^{\icx r\psi}\;,
\end{equation}
where $\set{f_p(y)}_{p\in\Z}$ is an arbitrary orthonormal basis of
$L^2([-1,1],\6y)$ (a Fourier basis will do). Since 
\begin{equation}
\label{gap04}
L_0 \ket{n,p,k,r} = \lambda^0_{n,k}\ket{n,p,k,r}\;, 
\qquad
\lambda^0_{n,k} = -n\gamma + \icx k\;,
\end{equation}
we see that $L_0$ has infinitely many eigenvalues on the imaginary
axis, which are infinitely degenerate. We can nevertheless apply standard
time-independent perturbation theory to second order in~$\kappa$ (see for
instance~\cite{LandauLifshitzVol3}). Expanding eigenfunctions in the
basis~\eqref{gap03}, then expanding everything in $\kappa$ and projecting on the
basis functions shows that the perturbed eigenfunctions have the form 
\begin{equation}
\label{gap05}
\ket{n,p,k,r} + \icx \kappa \sum_{m,q,l,s}
\frac{\braket{m,q,l,s}{L_1}{n,p,k,r}}{\lambda^0_{n,k}-\lambda^0_{m,l}}
\ket{m,q,l,s} + \Order{\kappa^2}\;,
\end{equation}
where the sum runs over all $(m,q,l,s)\neq(n,p,k,r)$ such that
$\braket{m,q,l,s}{L_1}{n,p,k,r}\neq0$. 
The perturbed eigenvalues have the form 
\begin{align}
\nonumber
\lambda_{n,p,k,r} ={}& \lambda^0_{n,k} 
+ \icx \kappa \braket{n,p,k,r}{L_1}{n,p,k,r} \\
&{}- \kappa^2 \sum_{m,q,l,s}
\frac{\abs{\braket{m,q,l,s}{L_1}{n,p,k,r}}^2}
{\lambda^0_{n,k}-\lambda^0_{m,l}} + \Order{\kappa^3}\;.
\label{gap06}
\end{align}
Note the minus sign due to the fact that the perturbation term $\icx L_1$ is
anti-hermitian. We now have to compute matrix elements of $L_1$ in the chosen
basis. We are however only interested in matrix elements involving $n=0$, which
give perturbations of eigenvalues on the imaginary axis. Using~\eqref{gap02B},
we get 
\begin{equation}
\label{gap07}
-\icx Z \sqrt{1-y^2}\,\sin\phi \dpar{}{y} \ket{0,p,k,r}
= \frac{\sigma}{\sqrt{2\gamma}} h_1(Z) \sqrt{1-y^2}\,f'_p(y) \frac12
\bigbrak{\e^{\icx(k+1)\phi}-\e^{\icx(k-1)\phi}} \e^{\icx r\psi}\;,
\end{equation}
showing that 
\begin{equation}
\label{gap08}
\braket{m,q,l,s}{-\icx Z \sqrt{1-y^2}\,\sin\phi \dpar{}{y}}{0,p,k,r} = 
\frac{\sigma}{\sqrt{2\gamma}} \delta_{m,1} a_{q,p} 
\bigbrak{\delta_{l,k+1}-\delta_{l,k-1}} \delta_{s,r}\;,
\end{equation}
where
\begin{equation}
\label{gap09}
a_{q,p} 
= \int_{-1}^1 \overline{f_q(y)} \sqrt{1-y^2} f'_p(y)\,\6y\;.
\end{equation}
Proceeding in a similar way for the other terms of $L_1$, we get 
\begin{equation}
\label{gap10}
\braket{m,q,l,s}{L_1}{0,p,k,r}
= \frac{\sigma}{\sqrt{2\gamma}} \delta_{m,1} 
\bigbrak{(a_{q,p}+b_{q,p})\delta_{l,k+1}
+(-a_{q,p}+b_{q,p})\delta_{l,k-1}} 
\delta_{s,r}\;,
\end{equation}
where
\begin{equation}
\label{gap11}
b_{q,p}(k,r)
= \int_{-1}^1 \overline{f_q(y)} \frac{ky+\frac12r}{\sqrt{1-y^2}} f_p(y)\,\6y\;.
\end{equation}
The expression~\eqref{gap10} of the matrix elements shows that the first-order
correction to the eigenvalues vanishes. It also shows that the second-order
correction term in~\eqref{gap06} is well-defined, because the denominator never
vanishes when the matrix element is nonzero. The expansion for the
eigenvalue becomes  
\begin{equation}
\label{gap12}
\lambda_{0,p,k,r} = \icx k
- \frac{(\kappa\sigma)^2}{2\gamma}  \sum_{q} \biggbrak{
\frac{\abs{a_{q,p}+b_{q,p}(k,r)}^2}{\gamma+\icx}
+\frac{\abs{a_{q,p}-b_{q,p}(k,r)}^2}{\gamma-\icx}
} + \Order{\kappa^3}\;.
\end{equation}
In particular, 
\begin{equation}
\label{gap13}
\re \lambda_{0,p,k,r} = - \frac{(\kappa\sigma)^2}{2(1+\gamma^2)}
\sum_{q} \Bigbrak{\abs{a_{q,p}+b_{q,p}(k,r)}^2 +
\abs{a_{q,p}-b_{q,p}(k,r)}^2} + \Order{\kappa^3}\;.
\end{equation}
It remains to see that the sum over $q$ does not vanish. Being a sum of
non-negative terms, it vanishes if and only if all its terms vanish, which
happens if and only if $a_{q,p}=0$ and $b_{q,p}(k,r)=0$ for all $q$. 
Now for $(k,r)\neq(0,0)$, $b_{q,p}(k,r)$ cannot vanish for all $q$, because the
function $(ky+\frac12r) f_p(y)/\sqrt{1-y^2}$ cannot be orthogonal to all basis
functions. For $(k,r)=(0,0)$, we have $b_{q,p}(0,0)=0$. However in that case,
$a_{q,p}=0$ for all $q$ if and only if $\sqrt{1-y^2} f'_p(y)$ is identically
zero, that is, only in the case $p=0$ where $f_p(y)$ is constant. But this
corresponds to the zero eigenvalue of the invariant measure.
\qed

%%%%%%%%%%%%%%%%%%%%%%%%%%%%%%%%%%%%%%%%%%%%%%%%%%%%%%%%%%%%%%%%%%%%%%%

\newpage
\section{Diffusion approximation}
\label{sec_ptrans}

In this section, it will be more convenient to consider both $\kappa$ and
$\sigma$ as small parameters.

%%%%%%%%%%%%%%%%%%%%%%%%%%%%%%%%%%%%%%%%%%%%%%%%%%%%%%%%%%%%%%%%%%%%%%%

\subsection{Averaging}
\label{ssec_avrg}

We consider the SDE on $\R\times[-1,1]\times\fS^1$ given by 
\begin{equation}
\label{avrg1}
\begin{split}
\6Z_t &= -\gamma Z_t\,\6t + \sigma\,\6W_t\;,\\
\6y_t &= 2\kappa Z_t\,\sqrt{1-y_t^2}\sin\phi_t\,\6t\;,\\
\6\phi_t &= \biggbrak{1 - 2\kappa Z_t\,\frac{y_t}{\sqrt{1-y_t^2}}
\cos\phi_t}  \6t\;.
\end{split}
\end{equation}
For small $\kappa$, the variable $\phi_t\simeq t$ changes much faster than
$y_t$, while $Z_t$ also fluctuates on timescales of order $1$. The 
philosophy of averaging tells us that the dynamics of $y_t$ should be
close to the one obtained by averaging the right-hand side over all fast
variables. However this average vanishes in the present case, so that we
have to look into the averaging procedure in more detail.

\begin{prop}
\label{prop_avrg}
There exists a change of variable $y\mapsto \bar y$, with $\bar
y\in[-1,1]$, such that 
\begin{equation}
\label{avrg02}
\6 \bar y_t = F(Z_t,\bar y_t,\phi_t)\,\6t
+ G(Z_t,\bar y_t,\phi_t)\,\6W_t\;,
\end{equation}
where the drift term satisfies 
\begin{equation}
\label{avrg03}
F(Z,\bar y,\phi) = -\frac{4\kappa^2\gamma}{1+\gamma^2} Z^2 \bar y 
+ \biggOrder{\frac{\kappa^3\gamma Z^3}{(1+\gamma^2)^2}}\;,
\end{equation}
and the diffusion term is of the form
\begin{equation}
\label{avrg04}
G(Z,\bar y,\phi) = 2\kappa\sigma
\sqrt{\frac{1-\bar y^2}{1+\gamma^2}}
\cos(\phi+\theta) 
+ \biggOrder{\frac{\kappa^2\sigma Z}{1+\gamma^2}}\;.
\end{equation}
where $\theta=\arctan\gamma$. 
\end{prop}

\begin{proof}
Averaging amounts to looking for a $\phi$-periodic change of variables
removing the dependence of $\6y_t$ on $\phi_t$. We thus set
$\bar y=y+\kappa w(Z,y,\phi)$ and plug this into the equation, yielding 
\begin{align}
\nonumber
\6\bar y_t ={}& \kappa\biggbrak{Z_t\sqrt{1-y_t^2}\sin\phi_t 
- \gamma Z_t\dpar wZ+ \dpar w\phi}\,\6t \\
&{}+ 2\kappa^2Z_t  \biggbrak{\sqrt{1-y_t^2}\sin\phi_t \dpar wy -
\frac{y_t}{\sqrt{1-y_t^2}}\cos\phi_t \dpar w\phi }\,\6t 
+ \kappa\sigma\dpar wZ \,\6W_t\;.
\label{avrg05}
\end{align}
The choice 
\begin{equation}
\label{avrg06}
w(Z,y,\phi) = \frac{2Z}{1+\gamma^2} \sqrt{1-y^2}
\bigbrak{\cos\phi + \gamma\sin\phi}
\end{equation}
allows to eliminate the first term in brackets, and yields
\begin{equation}
\label{avrg07}
\6\bar y_t = -\frac{4\kappa^2\gamma}{1+\gamma^2} Z_t^2 y_t \,\6t
+ \frac{2\kappa\sigma}{1+\gamma^2} \sqrt{1-y_t^2}
\bigbrak{\cos\phi_t + \gamma\sin\phi_t}
\,\6W_t\;.
\end{equation}
Finally, using the equality $\cos\phi + \gamma\sin\phi = \sqrt{1+\gamma^2}
\cos(\phi+\theta)$ and replacing $y_t$ by $\bar y_t$ yields the above
expressions for $\6\bar y_t$. 
\end{proof}

We see that the averaging transformation has created an effective drift
term, fluctuating rapidly in time.

%%%%%%%%%%%%%%%%%%%%%%%%%%%%%%%%%%%%%%%%%%%%%%%%%%%%%%%%%%%%%%%%%%%%%%%

\subsection{First-passage times}
\label{ssec_tau}

We would like to estimate the expected first-passage time of $\bar y_t$ at
any level $y<1$, and for the initial condition $(y_0,\phi_0)=(-1,0)$. We
split this problem into several parts, estimating first the time needed to
reach a level $-\delta_1<0$, then the level $0$, and finally all positive
$y$. Also, in order to avoid difficulties due to $Z_t$ becoming very large,
we will first work within the set $\setsuch{\omega}{\abs{Z_t}\leqs
1\;\forall t}$, and deal later with the rare events that $\abs{Z_t}$
becomes larger than~$1$. For this purpose, it is useful to introduce the
notation 
\begin{equation}
\label{tau02}
\tau_\cD = \inf\setsuch{t>0}{x_t \not\in \cD}
\end{equation}
for the first-exit time of $x_t=(Z_t,\bar y_t,\phi_t)$ from a set
$\cD\subset\R\times[-1,1]\times\fS^1$.
We denote by $\fP^{x_0}$ the law of the process starting in
some point $x_0=(Z_0,\bar y_0,\phi_0)$ at time $0$, and by $\E^{x_0}$
expectations with respect to this law.
Note that we have to consider here deterministic initial conditions $Z_0$, and
only at the end will we take expectations with respect to the initial
distribution of $Z$. 
We also introduce the notation 
\begin{equation}
\label{tau002}
T^\star = \frac{1+\gamma^2}{(\kappa\sigma)^2}\;.
\end{equation}
The following lemma will allow us to reduce the problem to time
intervals of fixed size~$T$. 

\begin{lemma}
\label{lem_tau0}
Fix $T>0$ and let 
\begin{equation}
\label{tau03}
q(T) = \sup_{x'_0\in\cD}
\bigprobin{x'_0}{\tau_\cD>T}\;.
\end{equation}
If $q(T)>0$, then 
\begin{equation}
\label{tau04}
\bigprobin{x_0}{\tau_\cD>t} \leqs q(T)^{-1} \e^{-\log q(T)^{-1} t/T}
\end{equation}
holds for all $t>0$ and all $x_0\in\cD$.
\end{lemma}
\begin{proof}
We set $t=nT+u$, with $n=\intpart{t/T}$. By the strong Markov property, we
have 
\begin{align}
\nonumber
\bigprobin{x_0}{\tau_\cD>t} 
&\leqs \bigprobin{x_0}{\tau_\cD>nT} \\
\nonumber
&\leqs \bigexpecin{x_0}{\indexfct{\tau_\cD>(n-1)T}
\probin{x_{(n-1)T}}{\tau_\cD>nT}} \\
\nonumber
&\leqs q(T) \, \bigprobin{x_0}{\tau_\cD>(n-1)T} \\
&\leqs \dots \leqs q(T)^n 
= e^{-n\log q(T)^{-1}}\;.
\label{tau05:1}
\end{align}
The result follows from the fact that $n\geqs (t/T)-1$.
\end{proof}

Note that this estimate implies immediately 
\begin{equation}
\label{tau06}
\expecin{x_0}{\tau_\cD} = \int_0^\infty \bigprobin{x_0}{\tau_\cD>t}\,\6t
\leqs \frac{T}{q(T)\log q(T)^{-1}}\;.
\end{equation} Our general strategy will be to find the smallest possible
$T$ such that $q(T)$ is of order $1$, i.e., independent of the parameters
$\kappa$, $\sigma$ and $\gamma$. Then $T$ will give the order of magnitude
of the first-exit time.

When deriving these estimates, we will often encounter the process
\begin{equation}
X_t = \frac{2\gamma}{\sigma^2} 
\frac{1}{t}\int_0^t Z_s^2\,\6s - 1\;. \\
\label{tau06A}
\end{equation}

\begin{lemma}
\label{lem_Xt}
For any initial condition $Z_0$, we have 
\begin{equation}
 \label{tau06B}
X_t = X_\infty + \overbar{X}_t,
\end{equation}
where $X_\infty$ has zero expectation, $\sqrt{X_\infty+1}$ follows a
standard gaussian distribution, and $\overbar{X}_t$ converges to zero in
probability. More precisely, for any $\eps>0$, 
\begin{equation}
\label{tau06C}
\bigprob{\abs{\overbar{X}_t}>\eps} =
\biggOrder{\frac{2\gamma \expec{Z_0^2}+\sigma^2}{2\gamma t\eps}}\;.
\end{equation}
\end{lemma}
\begin{proof}
Write $Z_t=Z^0_t + \e^{-\gamma t}(Z_0-Z^0_t)$, where $Z^0_t$ is a
stationary gaussian process of variance $\sigma^2/2\gamma$, and substitute
in~\eqref{tau06A}. The limit $X_\infty$ is obtained by applying the ergodic
theorem to the stationary part. The bound~\eqref{tau06C} is a consequence
of Markov's inequality applied to the remaining part.
\end{proof}

We start by characterising the time needed to go from $-1$ to a level
$-\delta_1$, where $\delta_1$ is a fixed constant in $(0,1)$. For an
interval $I\subset[-1,1]$, we set 
\begin{equation}
\label{tau07}
\cD_I 
= (-1,1)\times I \times \fS^1
\qquad\text{and}\qquad
\tau_I = \tau_{\cD_I}\;.
\end{equation}

\begin{prop}
\label{prop_tau01}
There exist constants $\sigma_1$, $c_1$ and
$\kappa_1(\delta_1)=\Order{\delta_1}$ such that whenever $\sigma<\sigma_1$
and $\kappa<\kappa_1$, 
\begin{equation}
\label{tau08}
\biggprobin{x_0}{\tau_{[-1,-\delta_1)} > 
c_1 \biggpar{\frac1\gamma \vee 
\frac{T^\star}{\delta_1}}} \leqs \frac34
\end{equation}
holds for all $x_0\in\cD_{[-1,-\delta_1)}$.
\end{prop}
\begin{proof}
For $t\leqs\tau_{[-1,-\delta_1)}$, we have $-\bar y_t\geqs \delta_1$
and thus 
\begin{equation}
\label{tau08:1}
F(Z_t,\bar y_t,\phi_t) 
\geqs \frac{4\kappa^2\gamma}{1+\gamma^2} Z_t^2
\Bigbrak{\delta_1-\Order{\kappa}}
\geqs \frac{2\kappa^2\gamma}{1+\gamma^2} \delta_1 Z_t^2
\end{equation}
provided $\kappa\leqs\kappa_1=\Order{\delta_1}$. By the comparison
principle for SDEs, we thus have 
\begin{equation}
\label{tau08:2}
\bar y_t - \bar y_0
\geqs \frac{(\kappa\sigma)^2\delta_1}{1+\gamma^2} t\bigpar{X_t+1} 
- \frac{2\kappa\sigma}{1+\gamma^2} Y_t\;,
\end{equation}
where we have introduced the martingale
\begin{equation}
Y_t = -\frac{1+\gamma^2}{2\kappa\sigma} 
\int_0^t G(Z_s,\bar y_s, \phi_s) \6W_s\;,
\label{tau08:3}
\end{equation}
whose variance of is bounded above by a constant times $t$.
Now we have, for any $T>0$,
\begin{align}
\label{tau08:5}
\bigprobin{x_0}{\tau_{[-1,-\delta_1)}>T} 
&\leqs \bigprobin{x_0}{\bar y_T < -\delta_1} \\
\nonumber
&\leqs \bigprobin{x_0}{X_T < -K} 
+ \Bigprobin{x_0}{Y_T \geqs \frac{\kappa\sigma}{2}\delta_1 T(1-K) 
+ \frac{1+\gamma^2}{2\kappa\sigma}(\delta_1+\bar y_0)} \;.
\end{align}
Lemma~\ref{lem_Xt} implies that 
\begin{align}
\nonumber
\bigprobin{x_0}{X_T < -K} 
&\leqs \Bigprobin{x_0}{X_\infty < 
-K+\sigma^2} + \Bigprobin{x_0}{\overbar{X}_T<-\sigma^2} \\
&= \Phi\Bigpar{\sqrt{1-K+\sigma^2}\,}-\Phi\Bigpar{-\sqrt{1-K+\sigma^2}\,} 
+ \biggOrder{\frac{2\gamma+\sigma^2}{2\gamma\sigma^2 T}}
\;,
\label{tau08:6}
\end{align}
where $\Phi(x)=(2\pi)^{-1/2}\int_{-\infty}^x \e^{-u^2/2}\6u$ denotes the
distribution function of the standard normal law. Taking $K$ sufficiently
close to $1$ ($K=15/16$ will do) and $\gamma T$ and $\sigma^2 T$
sufficiently large allows to make this probability smaller than $1/4$. On
the other hand, taking $\delta_1 T$ larger than
$2T^\star/(1-K)$ allows us to bound the second
term on the right-hand side of~\eqref{tau08:5} by $\probin{x_0}{ Y_T \geqs
0} = 1/2$.
\end{proof}

It follows immediately from Lemma~\ref{lem_tau0} that
$\probin{x_0}{\tau_{[-1,-\delta_1)}>t}$ decreases exponentially fast
with rate $\gamma\wedge\delta_1/T^\star$, and that 
\begin{equation}
\label{tau09}
\bigexpecin{x_0}{\tau_{[-1,-\delta_1)}} = \biggOrder{\frac1\gamma
\vee
\frac{T^\star}{\delta_1}}\;.
\end{equation}

We now turn to estimating the time needed for $\bar y_t$ to reach a
neighbourhood of zero. More precisely, we shall estimate
\begin{equation}
\label{tau091}
\bigprobin{x_0}{\tau_{(-1+\delta_2,-c^\star\kappa)} > t}
= \bigprobin{x_0}{-1+\delta_2 < \bar y_s < -c^\star\kappa, \abs{Z_s}<1
\;\forall s\leqs t}\;,
\end{equation}
where the constant $c^\star$, which is related to error terms
in~\eqref{avrg03}, will be defined below. Let us start by
considering, as a slight simplification of~\eqref{avrg02}, the linear
equation
\begin{equation}
\label{tau1}
\6y^0_t = -\frac{4\kappa^2\gamma}{1+\gamma^2} y^0_t\,\6t + G(Z_t,\bar
y_t,\phi_t)\,\6W_t\;,
\end{equation}
with initial condition $y^0_0=\bar y_0$. Let 
\begin{equation}
\label{tau2}
\alpha(t) = \frac{4\kappa^2\gamma}{1+\gamma^2} \int_0^t Z_s^2\,\6s
= \frac{2t}{T^\star} (X_t+1)\;.
\end{equation}
Then It\^o's formula shows that $y^0_t$ can be represented as 
\begin{equation}
\label{tau3}
y^0_t = \bar y_0 \e^{-\alpha(t)} + 
\e^{-\alpha(t)} \int_0^t \e^{\alpha(s)} G(Z_s,\bar y_s,\phi_s)\,\6W_s\;.
\end{equation}

\begin{lemma}
\label{lem_tau1}
Let $\tau_0=\inf\setsuch{t>0}{y^0_t}$ be the first time $y^0_t$
reaches $0$. Then
\begin{equation}
\label{tau4}
\bigprob{\tau_0>t} 
\leqs \frac{2\abs{\bar y_0}}{\sqrt{2\pi v(t)}}\;,
\end{equation}
where 
\begin{equation}
\label{tau5}
v(t) = \int_0^t \bigexpec{\e^{2\alpha(s)} G(Z_s,\bar y_s,\phi_s)^2}
\,\6s\;.
\end{equation}
\end{lemma}
\begin{proof}
When $\bar y_0=0$, the processes $y^0_t$ and $-y^0_t$ have the same
distribution. Thus by the strong Markov property, we can apply Andr\'e's
reflection principle, yielding,
for $\bar y_0>0$, 
\begin{equation}
\label{tau6:1}
\bigprob{\tau_0<t} = 2 \bigprob{y^0_t\leqs0}\;, 
\end{equation}
and thus 
\begin{equation}
\label{tau6:2}
\bigprob{\tau_0>t} = 2 \bigprob{y^0_t>0} - 1\;. 
\end{equation}
Let $Y_t = \int_0^t \e^{\alpha(s)} G(Z_s,\bar y_s,\phi_s)\,\6W_s$.
Being a stochastic integral, $Y_t$ is a Gaussian random variable, of
expectation $0$, while It\^o's formula shows that its variance is $v(t)$. 
Now observe that since $\e^{\alpha(t)}$ is a positive random variable, we
have 
\begin{align}
\nonumber
\bigprob{y^0_t>0} 
&= \bigprob{\bar y_0+Y_t > 0} = \bigprob{Y_t > -\bar y_0} \\
&= \Phi\biggpar{\frac{\bar y_0}{\sqrt{v(t)}}}
\leqs \frac12 + \frac{\bar y_0}{\sqrt{2\pi v(t)}}\;.
\label{tau6:3}
\end{align}
Inserting this
in~\eqref{tau6:2} yields the result.
\end{proof}

Before applying this to~\eqref{avrg02}, we need another technical lemma,
allowing to control the behaviour of $\phi_t$. 

\begin{lemma}
\label{lem_tau2}
Fix $\delta>0$ and an initial condition $x_0=(Z_0,\bar
y_0,\phi_0)\in\cD_{(-1+\delta,1-\delta)}$. For any $\phi^\star > \phi_0$
and $\kappa<\delta/2$, we have 
\begin{equation}
\label{tau10}
\tau_{\phi^\star} \defby
\inf\setsuch{t>0}{\phi_t=\phi^\star} \in
\biggbrak{\frac{\phi^\star-\phi_0}{1+2\kappa/\delta},
\frac{\phi^\star-\phi_0}{1-2\kappa/\delta}} \cup
[\tau_{(-1+\delta,1-\delta)},\infty]\;.
\end{equation}
\end{lemma}
\begin{proof}
Either $\tau_{\phi^\star}\geqs \tau_{(-1+\delta,1-\delta)}$ and we
are done, or $\tau_{\phi^\star}< \tau_{(-1+\delta,1-\delta)}$. In
the latter case,~\eqref{avrg1} implies that
$\brak{1-2\kappa/\delta}\6t \leqs \6\phi_t \leqs
\brak{1+2\kappa/\delta}\6t$, and the result follows by integrating from
$0$ to $\tau_{\phi^\star}$.
\end{proof}

\begin{prop}
\label{prop_tau02}
There exist constants $c^\star, c_2$
and $\kappa_2=\Order{\delta_2}$ such that whenever $\kappa<\kappa_2$, 
\begin{equation}
\label{tau11}
\biggprobin{x_0}{\tau_{(-1+\delta_2,-c^\star\kappa)} > t} \leqs 
\frac{c_2\abs{\bar y_0}}{\sqrt{\delta_2}}
\exp\biggset{-\frac{2t}{T^\star}}
\end{equation}
holds for all $x_0\in\cD_{(-1+\delta_2,-c^\star\kappa)}$ and all
$t\geqs T^\star$.
\end{prop}
\begin{proof}
For $t\leqs\tau_{(-1+\delta_2,1-\delta_2)}$, we have $\abs{Z_t}<1$
and thus 
\begin{equation}
\label{tau11:1}
F(Z_t,\bar y_t,\phi_t) 
\geqs -\frac{4\kappa^2\gamma}{1+\gamma^2} Z_t^2
\bigbrak{\bar y+c^\star\kappa}
\end{equation}
for some $c^\star>0$. By the comparison principle for SDEs, it follows that
$\bar y_t + c^\star\kappa \geqs y^0_t$, where $y^0_t$ obeys the
SDE~\eqref{tau1}. We can thus apply Lemma~\ref{lem_tau1}, and the proof
amounts to finding an exponentially growing lower bound for the variance
$v(t)$ defined in~\eqref{tau5}. 

Since for $t\leqs\tau_{(-1+\delta_2,-c^\star\kappa)}$, we have
$1-\bar y_t^2 \geqs 2\delta_2-\delta_2^2 \geqs \delta_2$, if follows that 
\begin{equation}
\label{tau11:3}
G(Z_t,\bar y_t,\phi_t)^2 \geqs
\frac{4\delta_2}{T^\star} 
\cos^2(\phi_t+\theta)
\end{equation}
for these times. We can thus write, using Jensen's inequality,  
\begin{align}
\nonumber
\frac{T^\star}{4\delta_2} v(t) 
&= \frac{T^\star}{4\delta_2}
\int_0^t \bigexpec{\e^{2\alpha(s)} G(Z_s,\bar y_s,\phi_s)^2}\,\6s \\
\nonumber
&\geqs \frac12
\int_0^t \indexfct{\cos^2(\phi_s+\theta)\geqs 1/2} 
\bigexpec{\e^{2\alpha(s)}}\,\6s
\\
\nonumber
&\geqs \frac12
\int_0^t \indexfct{\cos^2(\phi_s+\theta)\geqs 1/2} \e^{2
\bigexpec{\alpha(s)}}\,\6s
\\
&\geqs \frac12
\int_0^t \indexfct{\cos^2(\phi_s+\theta)\geqs 1/2}
\e^{4s/T^\star}\,\6s\;.
\label{tau11:4}
\end{align}
Lemma~\ref{lem_tau2} implies that (taking $\kappa\leqs\delta_2/4$)
during any time interval of length
$\pi/(1-2\kappa/\delta_2)$, $\cos^2(\phi_s+\theta)$ stays larger than $1/2$
for a
time span of length $\pi/2(1+2\kappa/\delta_2)$ at least. This implies that
\begin{align}
\nonumber
\int_{k\pi/(1-2\kappa/\delta_2)}^{(k+1)\pi/(1-2\kappa/\delta_2)}
\indexfct{\cos^2\phi_s\geqs 1/2}
\e^{4 s/T^\star}\,\6s
&\geqs \frac{T^\star}{4} 
\e^{4k\pi/T^\star(1-2\kappa/\delta_2)}
\bigbrak{\e^{2\pi/T^\star(1+2\kappa/\delta_2)}-1} \\
&\geqs \frac{\pi}{2(1+2\kappa/\delta_2)} 
\e^{4k\pi/T^\star(1-2\kappa/\delta_2)}\;.
\label{tau11:5}
\end{align}
This allows the variance $v(t)$ to be bounded below by a geometric series,
whose sum is of order $\delta_2(\e^{4t/T^\star}-1)$. The
result thus follows from Lemma~\ref{lem_tau1}. 
\end{proof}

Finally, we consider the problem of $\bar y_t$ reaching any level in
$(-1,1)$.

\begin{prop}
\label{prop_tau03}
For any $\delta_3>0$, there exist constants $c_3$, $h$ and
$\kappa_3=\Order{\delta_3}$ such that whenever $\kappa<\kappa_3$, 
\begin{equation}
\label{tau12}
\biggprobin{x_0}{\tau_{(-1+\delta_3,1-\delta_3)} > 
\frac{c_3}{\delta_3} \biggpar{\frac1\gamma
\vee T^\star}} \leqs 
\Phi\biggpar{\frac{1-\delta_3}{h\sqrt{\delta_3}}}
\end{equation}
holds for all $x_0\in\cD_{(-1+\delta_3,1-\delta_3)}$.
\end{prop}
\begin{proof}
As in the proof of Proposition~\ref{prop_tau02}, we can bound $\bar y_t$
below by $y^0_t-\kappa$, where $y^0_t$ is the solution of the linear
equation~\eqref{tau1}. Now for any $T, L>0$, we can write 
\begin{align}
\nonumber
\Bigprobin{x_0}{\sup_{0\leqs t\leqs T}y^0_t < L} 
\leqs{}& \biggprobin{x_0}{\inf_{T/2\leqs t\leqs T} \frac{\alpha(t)}{t}
\geqs\frac{2}{T^\star} } \\
&{}+ \biggprobin{x_0}{\sup_{T/2\leqs t\leqs T}y^0_t < L,\;
\exists t\in[T/2,T] \colon 
\alpha(t)<\frac{2t}{T^\star}}\;.
\label{tau12:1}
\end{align}
Using the process $X_t$ introduced in~\eqref{tau06A}, the first
term on the right-hand side can be rewritten as 
\begin{equation}
\label{tau12:2}
\biggprobin{x_0}{ 
\inf_{T/2\leqs t\leqs T} X_t \geqs 0}\;,
\end{equation}
which decreases exponentially fast, like $\e^{-\gamma T}$. 
Choosing $\gamma T \geqs \Order{1/\delta_3}$ makes this term of order
$\e^{-1/\delta_3}$. 
Let again $Y_t$ denote the stochastic integral $\int_0^t \e^{\alpha(s)}
G(Z_s,\bar y_s,\phi_s)\,\6W_s$. The
second term on the right-hand side of~\eqref{tau12:1} can be bounded by 
\begin{align}
\nonumber
&\biggprobin{x_0}{\sup_{T/2\leqs t\leqs T}\e^{-\alpha(t)}(\bar y_0+Y_t) <
L,
\;
\exists t\in[T/2,T] \colon
\alpha(t)<\frac{2t}{T^\star}} \\
\nonumber
&\quad\leqs
\biggprobin{x_0}{\exists t\in[T/2,T] \colon Y_t < L
\e^{2t/T^\star} - \bar y_0} \\
&\quad\leqs \sup_{t\in[T/2,T]}\Phi\biggpar{\frac{L
\e^{2t/T^\star} - \bar y_0}{\sqrt{v(t)}}}\;,
\label{tau12:3}
\end{align}
where $v(t)$ can be estimated with the help of~\eqref{tau11:4}. 
Taking $T$ sufficiently large that $\e^{-2T/T^\star}
\leqs \delta_3/2$ allows to bound $v(T/2)$ below by $(\delta_3/2)
\e^{2T/T^\star}$.
The argument of $\Phi$ can thus be bounded by a constant
times $1/\sqrt{\delta_3}$, while $L=1-\delta_3$. This yields the result.
\end{proof}

A partially matching lower bound is given by the following estimate.

\begin{prop}
\label{prop_tau04}
There exists a constant $c_4$ such that for any $L\in(0,1)$ and any
initial condition $x_0=(Z_0,\bar y_0,\phi_0)$ with
$\bar y_0>-L$,  
\begin{equation}
\label{tau13}
\biggprobin{x_0}{\inf_{0\leqs t\leqs T^\star} 
\bar{y}_t < -L} 
\leqs \exp\Bigset{-c_4\bigbrak{L+\bar y_0-c^\star\kappa}^2}\;.
\end{equation}
\end{prop}
\begin{proof}
Using the same notations as in the previous proof, and the fact that
$\alpha(t)\geqs0$, we can bound the probability by 
\begin{equation}
\label{tau14:1}
\biggprobin{x_0}{\sup_{0\leqs t\leqs T^\star} 
(-Y_t) > L + \bar y_0 - c^\star\kappa}\;.
\end{equation}
Since $Y_t$ is a martingale, a standard argument based on Doob's
submartingale inequality applied to $\e^{\eta Y_t}$
(cf.~\cite[Lemma~3.3]{BG6}) allows to bound this probability by the
exponential of $-c_4(L+\bar y_0 - c^\star\kappa)^2/v(T^\star)$. The
variance $v(T^\star)$ can be bounded above by a constant.
\end{proof}

%%%%%%%%%%%%%%%%%%%%%%%%%%%%%%%%%%%%%%%%%%%%%%%%%%%%%%%%%%%%%%%%%%%%%%%

\subsection{Renewal equations}
\label{ssec_renewal}

We now have to patch together the different estimates obtained in the
previous section in order to bound the expected transition time between
general domains in phase space. We shall do this by using renewal
equations, that we will solve using Laplace transforms. 
Recall that the Laplace transform of a random variable $\tau$ taking values 
in $\R_+$ is given by
\begin{equation}
\bigexpec{\e^{\lambda\tau}}
= \int_0^{\infty} \e^{\lambda t} \dpar{}{t}\prob{\tau\leqs t}\,\6t 
= 1 + \lambda \int_0^{\infty} \e^{\lambda t} \prob{\tau> t}\,\6t\;,
\label{renew01}
\end{equation}
provided $\lambda$ is such that $\lim_{t\to\infty}\e^{\lambda
t}\prob{\tau>t}=0$. If $\tau$ takes values in $\R_+\cup\set{\infty}$, we 
can also define
\begin{equation}
\bigexpec{\e^{\lambda\tau}\indexfct{\tau<\infty}}
= \prob{\tau<\infty} 
+ \lambda \int_0^{\infty} \e^{\lambda t} \prob{t<\tau<\infty}\,\6t\;,
\label{renew02}
\end{equation}
provided $\lambda$ is such that $\lim_{t\to\infty}\e^{\lambda
t}\prob{t<\tau<\infty}=0$. 

It is instructive to consider first the case of a one-dimensional process
$\set{y_t}_{t\geqs0}$, solution to a one-dimensional autonomous stochastic
differential equation. We assume that $y_t\in[-1,1]$ for all $t$, and
introduce levels $-1<h_1<h_2<h_3<1$ (\figref{fig3}a). We set
$\cD_1=[-1,h_2]$,
$\cD_2=[h_1,h_3]$ and introduce first-exit times $\tau_1=\tau_{\cD_1}$,
$\tau_2=\tau_{\cD_2}$, $\tau=\tau_{\cD_1\cup\cD_2}$ and 
\begin{equation}
\label{renew1}
\taudown = \inf\setsuch{t>0}{y_t<h_1}\;.
\end{equation}
The following result allows to express the expectation of $\tau$ using
informations on $\tau_1$, $\tau_2$ and $\taudown$. 

\begin{prop}
\label{prop_renewal1}
Assume there exists an $\eps>0$ such that the Laplace transforms
$\expecin{h_1}{\e^{\lambda\tau_1}}$ and $\expecin{h_2}{\e^{\lambda\tau_2}}$
are finite for all $\lambda<\eps$. Then 
\begin{equation}
\label{renew2}
\bigexpecin{h_1}{\tau} = 
\frac{\bigexpecin{h_1}{\tau_1}+\bigexpecin{h_2}{\tau_2}}
{\bigprobin{h_2}{\tau_2<\taudown}}\;.
\end{equation}
\end{prop}
\begin{proof}
Since $\tau>\tau_1$ for all initial conditions in $\cD_1$, 
we have, by the strong Markov property and time homogeneity, 
the renewal equation
\begin{align}
\nonumber
\bigprobin{h_1}{\tau>t} 
&= \bigprobin{h_1}{\tau_1>t} 
+ \bigexpecin{h_1}{\indexfct{\tau_1\leqs t}
\bigprobin{h_2,\tau_1}{\tau>t}} \\
&= \bigprobin{h_1}{\tau_1>t} + \int_0^t \dpar{}{s}
\bigprobin{h_1}{\tau_1\leqs s} \bigprobin{h_2}{\tau>t-s} \,\6s\;.
\label{renew3:1}
\end{align}
Taking the Laplace transform of~\eqref{renew3:1} and
making the change of variables $t-s=u$ shows that 
\begin{align}
\nonumber
\expecin{h_1}{\e^{\lambda\tau}}
&= \expecin{h_1}{\e^{\lambda\tau_1}}
+ \lambda\int_0^{\infty} \e^{\lambda s} 
\dpar{}{s}\probin{h_1}{\tau_1\leqs s}\,\6s 
\int_0^\infty \e^{\lambda u} \probin{h_2}{\tau>u}\,\6u \\
\nonumber
&= \expecin{h_1}{\e^{\lambda\tau_1}} +
\expecin{h_1}{\e^{\lambda\tau_1}}\bigpar{\expecin{h_2}{\e^{\lambda\tau}}-1}\\
&= \expecin{h_1}{\e^{\lambda\tau_1}}\expecin{h_2}{\e^{\lambda\tau}}\;.
\label{renew3:3}
\end{align}
Now since for initial conditions in $\cD_2$, $\tau_2=\tau\wedge\taudown$,
we also have the renewal equation 
\begin{equation}
\bigprobin{h_2}{\tau>t} 
= \bigprobin{h_2}{\tau_2>t} 
+ \bigexpecin{h_2}{\indexfct{\taudown\leqs t\wedge\tau_2}
\bigprobin{h_1,\taudown}{\tau>t}}\;. 
\label{renew3:4}
\end{equation}
The condition $\taudown\leqs t\wedge\tau_2$ can be rewritten as
$\taudown^\#\leqs t$, where $\taudown^\#\in\R_+\cup\set{\infty}$ denotes
the first time the process $y_t$, killed upon reaching $h_3$, reaches
$h_1$ (we set $\taudown^\#=\infty$ in case this never happens). Taking the
Laplace transform, we get 
\begin{equation}
\label{renew3:5}
\expecin{h_2}{\e^{\lambda\tau}}
= \expecin{h_2}{\e^{\lambda\tau_2}} 
+ \bigexpecin{h_2}{\e^{\lambda\taudown^\#}\indexfct{\taudown^\#<\infty}}
(\expecin{h_1}{\e^{\lambda\tau}}-1)\;.
\end{equation}
Combining~\eqref{renew3:3} with~\eqref{renew3:5}, we obtain 
\begin{equation}
\label{renew3:6}
\expecin{h_1}{\e^{\lambda\tau}} = 
\frac{\expecin{h_1}{\e^{\lambda\tau_1}}
\Bigpar{\expecin{h_2}{\e^{\lambda\tau_2}} - 
\bigexpecin{h_2}{\e^{\lambda\taudown^\#} \indexfct{\taudown^\#<\infty}}}}
{1-\expecin{h_1}{\e^{\lambda\tau_1}}
\bigexpecin{h_2}{\e^{\lambda\taudown^\#} \indexfct{\taudown^\#<\infty}}}\;.
\end{equation}
The result then follows by taking the derivative of this relation with
respect to $\lambda$, evaluating it in $\lambda=0$, and using the fact
that $\probin{h_2}{\taudown^\#<\infty}=\probin{h_2}{\taudown<\tau_2}$.
\end{proof}

\begin{remark}
\label{rem_renew1}
Relation~\eqref{renew3:6} shows that we do in fact control the Laplace transform,  
and thus the distribution of $\tau$. Since 
\begin{equation}
\label{renew4A}
\bigprob{\tau>t} = \bigprob{\e^{\lambda\tau}>\e^{\lambda t}}
\leqs \e^{-\lambda t} \bigexpec{\e^{\lambda\tau}}
\end{equation}
for all $\lambda$ such that $\expec{\e^{\lambda\tau}}$ is bounded, we see 
that the tails of the distribution of $\tau$ decay exponentially fast, with 
a rate controlled by the domain of the Laplace transform. This domain depends 
on the rate of exponential decay of $\tau_1$ and $\tau_2$ and on the denominator 
in~\eqref{renew3:6}.
\end{remark}

Relation~\eqref{renew2} has a slightly asymmetric form, due to the fact
that in the situation considered, we always have $\tau_1<\tau$. It can
however easily be extended to situations where this relation does not
hold. We will consider the still more general situation of a
time-homogeneous SDE in $\R^n$. Let $\cD_1$ and $\cD_2$ be two open sets
in $\R^n$, and let $\tau_1$, $\tau_2$ and $\tau$ denote the first-exit
times from $\cD_1$, $\cD_2$ and $\cD_1\cup\cD_2$ respectively. Finally let
$\cC_1=\partial\cD_1\cap\cD_2$ and $\cC_2=\partial\cD_2\cap\cD_1$
(\figref{fig3}b).

\begin{figure}
\centerline{\includegraphics*[clip=true,height=45mm]{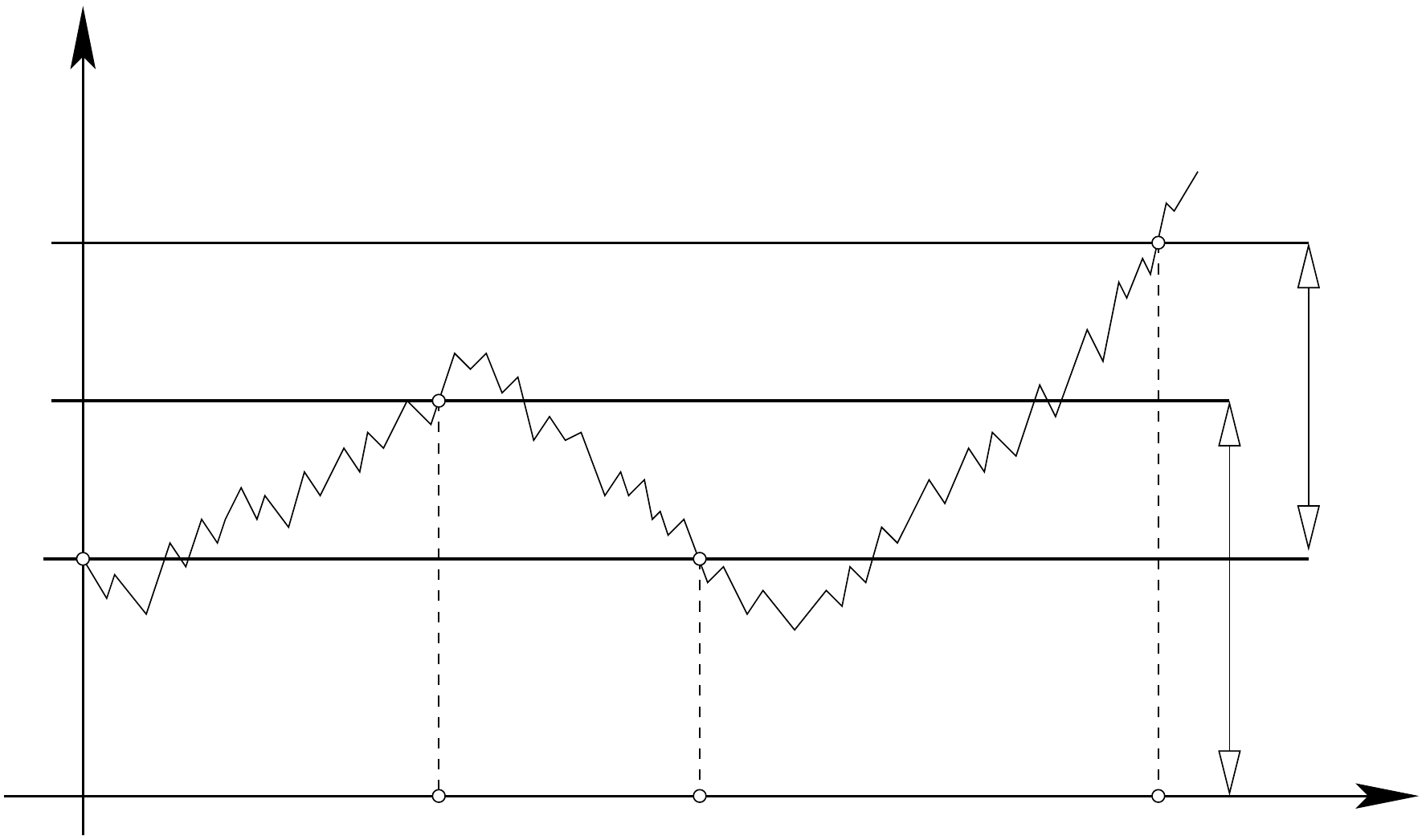}
\hspace{15mm}
\includegraphics*[clip=true,height=50mm]{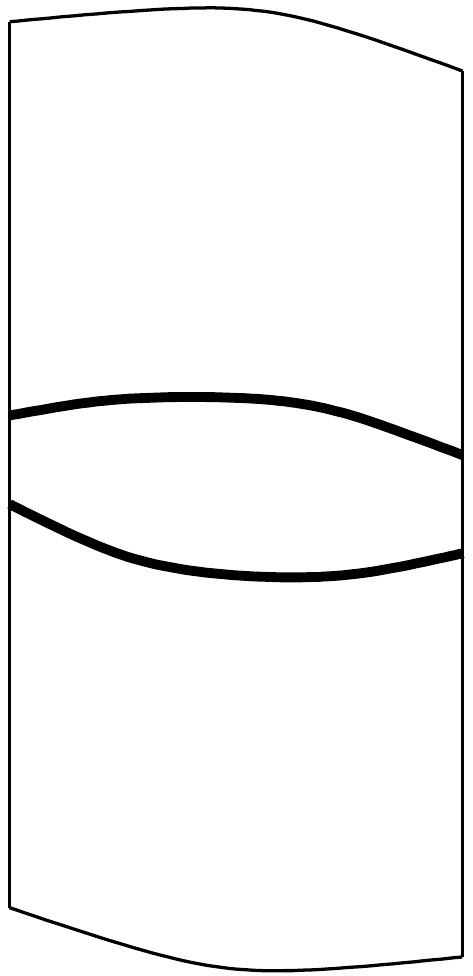}}
 \figtext{ 
	\writefig	0.5	5.2	{\bf (a)}
	\writefig	8.6	0.4	$t$
	\writefig	3.7	0.4	$\tau_1$
	\writefig	5.1	0.4	$\taudown$
	\writefig	7.6	0.4	$\tau$
	\writefig	8.2	1.4	$\cD_1$
	\writefig	8.6	2.8	$\cD_2$
	\writefig	8.0	4.1	$\bar y_t$
	\writefig	1.6	4.4	$\bar y$
	\writefig	1.35	1.9	$h_1$
	\writefig	1.35	2.8	$h_2$
	\writefig	1.35	3.6	$h_3$
	\writefig	9.5	5.2	{\bf (b)}
	\writefig	12.7	0.7	$\cD_1$
	\writefig	12.7	4.9	$\cD_2$
	\writefig	11.8	3.6	$\cC_1$
	\writefig	11.8	2.2	$\cC_2$
	\writefig	11.4	2.9	$\cD_1\cap\cD_2$
 }
 \vspace{2mm}
\caption[]{{\bf (a)} The setting of Proposition~\ref{prop_renewal1}. The
first-exit time $\tau_2$ from $\cD_2$ being the minimum
$\taudown\wedge\tau$, the distribution of $\tau$ can be related to the
distributions of  $\tau_1$, $\tau_2$ and $\taudown$. 
{\bf (b)} The setting
of Proposition~\ref{prop_renewal2}. 
}
\label{fig3}
\end{figure}

\begin{prop}
\label{prop_renewal2}
Assume there exists an $\eps>0$ such that the Laplace transforms
of $\tau_1$ and $\tau_2$ are finite for all $\lambda<\eps$. Then 
\begin{equation}
\label{renew4}
\sup_{x_0\in\cD_1\cup\cD_2}\bigexpecin{x_0}{\tau} \leqs  
\frac{2 \displaystyle 
\biggpar{\,\sup_{x\in\cD_1}\bigexpecin{x}{\tau_1}
+\sup_{y\in\cD_2}\bigexpecin{y}{\tau_2}}}
{1 - \displaystyle
\sup_{x\in\cC_2}\bigprobin{x}{x_{\tau_1}\in\cC_1}
\sup_{y\in\cC_1}\bigprobin{y}{x_{\tau_2}\in\cC_2}}\;.
\end{equation}
\end{prop}
\begin{proof}
Let $\tau^\#_1$ denote the first-exit time from $\cD_1$ of the process
killed upon hitting $\partial\cD_1\setminus\cD_2$, i.e., we set
$\smash{\tau^\#_1}=\infty$ in case $x_t$ hits $\partial\cD_1\setminus\cD_2$
before $\cC_1$. Then we have, for initial conditions $x\in\cD_1$, 
\begin{equation}
\label{renew5:1}
\bigprobin{x}{\tau>t}
= \bigprobin{x}{\tau_1>t}
+ \bigexpecin{x}{\indexfct{\tau^\#_1<t}
\bigprobin{x_{\smash{\tau^\#_1}}}{\tau>t-\tau^\#_1}}\;.
\end{equation}
Taking the Laplace transform, and using the change of variables
$t=\tau^\#_1+s$, we get 
\begin{equation}
\label{renew5:2}
\bigexpecin{x}{\e^{\lambda\tau}}
= \bigexpecin{x}{\e^{\lambda\tau_1}}
+ \lambda\int_0^\infty \e^{\lambda s} 
\Bigexpecin{x}{\e^{\lambda\tau^\#_1}\indexfct{\tau^\#_1<\infty}
\bigprobin{x_{\smash{\tau^\#_1}}}{\tau>s}}\,\6s\;.
\end{equation}
Evaluating the derivative in $\lambda=0$ yields the bound 
\begin{equation}
\label{renew5:3}
\bigexpecin{x}{\tau} \leqs \bigexpecin{x}{\tau_1} 
+ \sup_{y_\in\cC_1} \bigexpecin{y}{\tau}
\bigprobin{x}{\tau^\#_1<\infty}\;.
\end{equation}
A similar relation holds for initial conditions $y\in\cD_2$. Combining
both relations after taking the supremum over $x\in\cC_2$ and $y\in\cC_1$,
we get 
\begin{equation}
\label{renew5:4}
\sup_{x\in\cC_2} \bigexpecin{x}{\tau} 
\leqs 
\frac{\displaystyle
\sup_{x\in\cC_2}\bigexpecin{x}{\tau_1} +
\sup_{y\in\cC_1}\bigexpecin{y}{\tau_2}
\sup_{x\in\cC_2}\bigprobin{x}{\tau^\#_1<\infty}}
{1 - \displaystyle
\sup_{x\in\cC_2}\bigprobin{x}{\tau^\#_1<\infty}
\sup_{y\in\cC_1}\bigprobin{y}{\tau^\#_2<\infty}}\;,
\end{equation}
and a similar relation for initial conditions in $\cC_1$. Using
again~\eqref{renew5:3} and its equivalent, but now for general
$x_0\in\cD_1$ and $y_0\in\cD_2$, yields the result (after some
simplifications -- the upper bound is not sharp).
\end{proof}

%%%%%%%%%%%%%%%%%%%%%%%%%%%%%%%%%%%%%%%%%%%%%%%%%%%%%%%%%%%%%%%%%%%%%%%

\subsection{Proof of Theorem~\ref{thm_res2}}
\label{ssec_thm2}

We can now proceed to the proof of Theorem~\ref{thm_res2}, characterising
the expectation of the first-passage time 
\begin{equation}
\label{tm01}
\tau(y) = \inf\setsuch{t>0}{\bar y_t>y}
\end{equation}
of $\bar y_t$ at any level $y\in[-1,1]$. Notice that we have (cf.~\eqref{tau07})
\begin{equation}
\label{tm02}
\tau_{[-1,y)} = \tau(y) \wedge \tau_Z\;,
\end{equation}
where 
\begin{equation}
\label{tm03}
\tau_Z = \inf\setsuch{t>0}{\abs{Z_t}>1}
\end{equation}
is typically very large for small $\sigma$. 
We start by bounding the expectation of $\tau_{[-1,y)}$.

\begin{prop}
\label{prop_thm2:1}
For any $x_0=(Z_0,\bar y_0,\phi_0)$ with $\abs{Z_0}\leqs 1$, 
\begin{equation}
\label{tm1}
\bigexpecin{x_0}
{\tau_{[-1,y)}} 
\leqs e_1(y)
\biggbrak{\frac1\gamma\vee T^\star}\;,
\end{equation}
where the function $e_1(y)$ is bounded for $y<1$. 
\end{prop}
%\goodbreak
\begin{proof} %\hfill
% \begin{enum}
% \item	
We start by setting $\cD_1=\cD_{[-1,-\delta_1)}$ and 
$\cD_2=\cD_{(-1+\delta_2,-c^\star\kappa)}$. Then
Proposition~\ref{prop_renewal2} yields a bound of the form
\begin{equation}
\label{tm2:1}
\bigexpecin{x_0}{\tau_{[-1,-c^\star\kappa)}}
\leqs \frac{2(E_1+E_2)}{1-P_1P_2}\;,
\end{equation}
where Proposition~\ref{prop_tau01} shows 
\begin{equation}
\label{tm2:2}
E_1 = \sup_{x\in\cD_{[-1,-\delta_1)}}
\bigexpecin{x}{\tau_{[-1,-\delta_1)}}
= \biggOrder{\frac1\gamma\vee\frac{T^\star}{\delta_1}}\;,
\end{equation}
and Proposition~\ref{prop_tau02} yields 
\begin{equation}
\label{tm2:3}
E_2 = \sup_{x\in\cD_{(-1+\delta_2,-c^\star\kappa)}}
\bigexpecin{x}{\tau_{(-1+\delta_2,-c^\star\kappa)}}
\leqs \frac{c_2}{\sqrt{\delta_2}}T^\star\;. 
\end{equation}
We expect $P_1$ to be very close to $1$, so we only try to bound $P_2$,
which is done as follows. Let $\taudown$ and $\tauup$ denote the
first-hitting times of the levels $-1+\delta_2$ and $-c^\star\kappa$
respectively. Then 
\begin{align}
\nonumber
P_2 &= \sup_{x \colon \bar y=-\delta_1}
\bigprobin{x}{\bar y_{\tau_{[-1,-\delta_1)}}=-1+\delta_2} \\
\nonumber
&= \sup_{x \colon \bar y=-\delta_1}
\bigprobin{x}{\taudown<\tauup} \\
&\leqs \sup_{x \colon \bar y=-\delta_1}
\Bigbrak{\bigprobin{x}{\taudown<T^\star}
+ \bigprobin{x}{\tauup>T^\star}}\;.
\label{tm2:4}
\end{align}
Proposition~\ref{prop_tau04} shows that $\probin{x}{\taudown<T^\star}$ is
of order $\e^{-c_4(1-\delta_2-2c^\star\kappa)^2}$, while
Proposition~\ref{prop_tau02} shows that $\probin{x}{\tauup>T^\star}$ is of
order $\abs{\delta_1}/\sqrt{\delta_2}$. Thus choosing, for instance,
$\delta_2=1/2$, and $\delta_1$ small enough makes $P_2$ bounded away from
one. This proves~\eqref{tm1} for all $y<-c^\star\kappa$.
We now repeat the procedure, going first up to level $1/2$,
and finally to level $1-\delta_3$. The $\delta_i$'s can be chosen each
time in such a way that $\prob{\taudown<\tauup}$ is bounded away from $1$.
% \qed
% \end{enum}
% \renewcommand{\qed}{}
\end{proof}

% We can now complete the proof of Theorem~\ref{thm_res2}.

\begin{proof}[{\sc Proof of Theorem~\ref{thm_res2}}]
We will use the fact that the Ornstein--Uhlenbeck process $Z_t$ 
reaches $0$ in a time of order $1/\gamma$, and needs an exponentially 
long time to go from $0$ to $1$. In order to apply Proposition~\ref{prop_renewal2}, 
we introduce sets 
\begin{align}
\nonumber
\cD_1 &= (-1,1)\times[-1,y)\times\fS^1 = \cD_{[-1,y)} \\
\cD_2 &= \R^*\times(-1,y)\times\fS^1\;.
\label{tm4:1}
\end{align}
Then the first-exit time from $\cD_2$ equals $\tau_0\wedge\tau(y)$, 
where $\tau_0$ denotes the first time $Z_t$ hits $0$. We have 
$\cC_1=\partial\cD_1\cap\cD_2=\set{-1,1}\times(-1,y)\times\fS^1$ 
and $\cC_2=\partial\cD_2\cap\cD_1=\set{0}\times(-1,y)\times\fS^1$.
Thus~\eqref{renew5:4} yields 
\begin{equation}
\label{tm4:3}
\sup_{x\in\cC_2}\bigexpecin{x}{\tau(y)} 
\leqs \frac{\displaystyle \sup_{x\in\cC_2} 
\bigexpecin{x}{\tau_{[-1,y)}} + 
\sup_{y\in\cC_1} \bigexpecin{y}{\tau_0}
\sup_{x\in\cC_2} \bigprobin{x}{\tau_Z\leqs\tau(y)}}
{1 - \displaystyle 
\sup_{x\in\cC_2} \bigprobin{x}{\tau_Z\leqs\tau(y)}
\sup_{y\in\cC_1} \bigprobin{y}{\tau_0\leqs\tau(y)}
}\;.
\end{equation}
Using the reflection principle, one obtains 
$\expecin{y}{\tau_0} = \Order{1/\gamma}$. Furthermore, we can 
write for any $K>0$
\begin{equation}
\label{tm4:4}
\bigprobin{x}{\tau_Z\leqs\tau(y)} 
\leqs \bigprobin{x}{\tau_Z\leqs K}
+ \bigprobin{x}{\tau(y)\wedge\tau_Z > K}\;.
\end{equation}
We choose $K = N(\gamma^{-1}\vee T^\star)$ for some $N>0$. 
Then on one hand, Markov's inequality,~\eqref{tm02} and 
Proposition~\ref{prop_thm2:1} yield 
\begin{equation}
\label{tm4:5}
\bigprobin{x}{\tau(y)\wedge\tau_Z > K} 
 = \bigprobin{x}{\tau_{[-1,y)} > K} 
\leqs \frac{\bigexpecin{x_2}{\tau_{[-1,y)}}}{K}
\leqs \frac{e_1(y)}{N}\;.
\end{equation}
On the other hand, a well-known property of the Ornstein--Uhlenbeck process 
is that $\bigprobin{x}{\tau_Z\leqs K}$ behaves like $(K/\sigma)\e^{-\gamma/\sigma^2}$ 
(see, for instance, \cite[Proposition~3.3]{BG1} or \cite[Appendix]{BG7}).
Thus choosing for instance $N^2=e_1(y)\sigma\e^{\gamma/\sigma^2}/(\gamma^{-1}\vee T^\star)$
makes $\probin{x}{\tau_Z\leqs\tau(y)}$ of order $\sqrt{e_1(y)(\gamma^{-1}\vee T^\star)} 
\sigma^{-1/2}\e^{-\gamma/2\sigma^2}$, which is negligible compared to the expectation 
of $\tau_{[-1,y)}$ and to $1$. 
This proves the result for initial conditions $x\in\cC_2$. 

In order to extend it to general initial conditions, it suffices to 
apply~\eqref{renew5:3} and its twin relation, and 
to integrate over the initial distribution of $Z$.
\end{proof}

%%%%%%%%%%%%%%%%%%%%%%%%%%%%%%%%%%%%%%%%%%%%%%%%%%%%%%%%%%%%%%%%%%%%%%%

\subsection*{Acknowledgements}

This work has been supported by the French Ministry of Research, by way of
the {\it Action Concert\'ee Incitative (ACI) Jeunes Chercheurs,
Mod\'elisation stochastique de syst\`emes hors \'equilibre\/}.

\small
\bibliography{../SU2}
\bibliographystyle{amsalpha}               

\goodbreak
\bigskip\bigskip\noindent
{\small 
Jean-Philippe Aguilar \\ 
Centre de Physique Th\'eorique (CPT)\\
{\sc CNRS, UMR 6207} \\
Campus de Luminy, Case 907 \\
13288~Marseille Cedex 9, France \\
{\it E-mail address: }{\tt aguilar@cpt.univ-mrs.fr}
}

\bigskip\bigskip\noindent
{\small 
Nils Berglund \\ 
Universit\'e d'Orl\'eans, Laboratoire {\sc MAPMO} \\
{\sc CNRS, UMR 6628} \\
F\'ed\'eration Denis Poisson, FR 2964 \\
B\^atiment de Math\'ematiques, B.P. 6759\\
45067~Orl\'eans Cedex 2, France \\
{\it E-mail address: }{\tt nils.berglund@univ-orleans.fr}

}

%%%%%%%%%%%%%%%%%%%%%%%%%%%%%%%%%%%%%%%%%%%%%%%%%%%%%%%%%%%%%%%%%%%%%%%%%%%%%%

\end{document}